\documentclass{amsart}
\usepackage{amsmath}
\usepackage{amssymb}
\usepackage{amsfonts}
\usepackage{mathtools}
\usepackage{graphicx}
\usepackage{subfigure}
\usepackage{float}
\usepackage{bbm}
\usepackage{amsthm}
\usepackage{hyperref}
\usepackage{cleveref}
\usepackage{tikz}
\usepackage{tikz-3dplot}
\usepackage{mathrsfs}
\usepackage{bm}
\usepackage{esint}
\usepackage{appendix}
\usepackage{xfrac}

\numberwithin{table}{section}
\numberwithin{equation}{section}

\newtheorem{theorem}[equation]{Theorem}
\newtheorem{corollary}[equation]{Corollary}  
\newtheorem{lemma}[equation]{Lemma} 
\newtheorem{proposition} [equation]{Proposition}

\theoremstyle{definition}
\newtheorem{definition}[equation]{Definition}

\theoremstyle{remark}
\newtheorem{remark}[equation]{Remark}

\title{Metric Degenerations and Satake Compactification of Enriques Surfaces}

\author{Zexuan Ouyang}
\address{Westlake Institute for Advanced Study, Westlake University, Hangzhou 310030, China}
\email{ouyangzexuan@westlake.edu.cn}

\date{}

\newcommand{\IC}{{\mathbb C}}

\newcommand{\IP}{{\mathbb P}} 
\newcommand{\IQ}{{\mathbb Q}} 
\newcommand{\IR}{{\mathbb R}}

\newcommand{\IZ}{{\mathbb Z}}

\newcommand{\CM}{{\mathcal M}}

\newcommand{\FM}{{\mathfrak M}}

\newcommand{\en}{_{\mathrm {En}}}

\begin{document}
	
	\bibliographystyle{plain}

	\begin{abstract}
		We determine the Gromov–Hausdorff compactification of the moduli space of unit-diameter Ricci-flat Kähler metrics on Enriques surfaces.
		We show that all metric degenerations are encoded by the adjoint Satake compactification of the Enriques period domain.
		There are $32$ boundary components added in this Satake compactification, and the GH limit at each boundary is determined. As a corollary, all the GH limits of K\"ahler Ricci-flat metrics on Enriques surfaces are classified, including those for a fixed complex structure or a fixed polarization.
	\end{abstract}
	
	\maketitle
	
	\setcounter{tocdepth}{1}
	\tableofcontents
	
	\section{Introduction}
	
	A fundamental problem in the study of Ricci-flat K\"ahler metrics is to understand
	their degenerations and to relate the resulting metric compactification to
	algebraic or arithmetic compactifications of the corresponding moduli spaces.
	The degeneration
	of Ricci-flat metrics has been studied from several complementary points of
	view, including the collapse of Calabi--Yau metrics
	\cite{gross-tosatti-zhang-16,gross-wilson-00}, and compactifications of
	period and moduli spaces \cite{anderson-92,kobayashi-todorov-87}. For K3 surfaces, the relation between metric
	degenerations and arithmetic compactifications has been investigated in
	\cite{odaka-oshima-21,sun-zhang-24,
		oy-collapsingk3surfacesspecial,
		ouyang2025compactificationmetricmodulispace}.
	For Enriques surfaces, the arithmetic geometry of the period space and its
	compactifications has a classical history; see, for instance,
	\cite{sterk,enriques-book}.
	
	In this paper we determine the Gromov--Hausdorff compactification of the
	moduli space of unit-diameter Ricci-flat K\"ahler metrics on Enriques
	surfaces. We show that it is completely encoded by an adjoint Satake
	compactification of the Enriques period domain. The arithmetic boundary
	consists of exactly $32$ components, and the Gromov--Hausdorff limit
	associated with every boundary component is determined explicitly.
	
	The problem is intrinsically equivariant. Every Enriques surface $Y$ has a
	K3 universal cover
	\[
	p:X\longrightarrow Y
	\]
	with a fixed-point-free holomorphic involution $\sigma$, and every Ricci-flat
	K\"ahler metric on $Y$ lifts to a $\sigma$-invariant hyperk\"ahler metric on
	$X$. Under degeneration, however, the limiting involution on the K3
	Gromov--Hausdorff limit need not remain free. Fixed points or
	positive-dimensional fixed loci may appear, and different limiting
	involutions can produce different quotient spaces. Thus the Enriques
	compactification is not determined solely by the underlying K3 limit: the
	limiting involution is an essential part of the boundary geometry.
	
	We now describe the moduli space. Let $\Lambda$ denote the K3 lattice and let
	$\rho$ be the lattice involution induced by the Enriques covering involution.
	The corresponding invariant and anti-invariant lattices have signatures
	$(1,9)$ and $(2,10)$, respectively; see
	\cite{barth-cptcplxsurface-04,enriques-book}. Set
	\[
	G=O(1,9)\times O(2,10),
	\qquad
	G_{\IZ}=\{g\in\operatorname{Aut}(\Lambda):g\rho=\rho g\},
	\]
	and let $K$ be a maximal compact subgroup of $G$. The period space relevant
	to Ricci-flat Enriques metrics is the arithmetic quotient
	\[
	\CM\en=G_{\IZ}\backslash G/K.
	\]
	After removing the root hyperplanes corresponding to singular data, one obtains an open subset
	$\CM\en^0\subset\CM\en$. By the period map and the Calabi--Yau theorem
	\cite{yau-78}, there is a bijection
	\begin{equation}
		\label{Phi-1.1}
		\Phi^0:\CM\en^0\longrightarrow\FM\en^{sm},
	\end{equation}
	where $\FM\en^{sm}$ denotes the moduli space of smooth unit-diameter
	Ricci-flat K\"ahler metrics on Enriques surfaces.
	
	Our main result identifies the metric degeneration with the adjoint Satake
	boundary.
	
	\begin{theorem}\label{thm:main-compactification}
		There exists a continuous geometric realization map
		\[
		\overline{\Phi}:
		\overline{\CM}\en^{\mathrm{Sat,ad}}
		\longrightarrow
		\overline{\FM}\en^{GH},
		\]
		extending the map $\Phi^0$ in \eqref{Phi-1.1}.
		Moreover, $\overline{\CM}\en^{\mathrm{Sat,ad}}$ has exactly $32$
		arithmetic boundary components, and the Gromov--Hausdorff limit
		associated with every boundary component is determined explicitly;
		see Theorems~\ref{thm:isotropic_classes}
		and~\ref{thm-quotient-type}.
	\end{theorem}
	
	The resulting quotient geometries are richer than those of the K3 covers
	themselves. In particular, the topology of a $2$-dimensional Enriques limit
	depends on the limiting involution.
	
	\begin{corollary}\label{cor:all-gh-limits}
		Let $Y$ be an Enriques surface equipped with a family of unit-diameter
		K\"ahler Ricci-flat metrics. The Gromov--Hausdorff limits are classified
		according to their real dimension as follows:
		\begin{itemize}
			\item \textbf{3-dimensional limits:}
			the limit is a $\IZ_2$-quotient of $T^3/\IZ_2$, with the quotient
			action explicitly determined by the corresponding Satake boundary
			component;
			
			\item \textbf{2-dimensional limits:}
			the limit is homeomorphic to
			\[
			\IP^1,\qquad \IR\IP^2,\qquad\text{or}\qquad D^2;
			\]
			
			\item \textbf{1-dimensional limits:}
			the limit is the unit interval $I^1$.
		\end{itemize}
	\end{corollary}
	
	The global compactification also gives a classification of degenerations
	under natural geometric constraints.
	
	\begin{corollary}\label{coro-fix-classes}
		The Gromov--Hausdorff limits of K\"ahler Ricci-flat metrics on Enriques
		surfaces under the following restrictions are classified as follows:
		\begin{itemize}
			\item \textbf{Fixed complex structure:}
			let $(Y,J)$ be fixed and vary the K\"ahler class. The only possible
			collapsing limit is $\IP^1$, equipped with the generalized
			K\"ahler--Einstein metric associated with an elliptic fibration
			\[
			f:Y\longrightarrow\IP^1.
			\]
			Conversely, every such metric is realized as a
			Gromov--Hausdorff limit.
			
			\item \textbf{Fixed polarized K\"ahler class:}
			if the polarized K\"ahler class is fixed and the complex structure
			varies, the only possible collapsing limit spaces are
			\[
			\IR\IP^2,\qquad D^2,\qquad I^1.
			\]
		\end{itemize}
	\end{corollary}

	This can be read directly from the global compactification:
	fixing the complex structure allows only the $O(1,9)$ factor of the period
	domain to diverge, whereas fixing the polarized K\"ahler class allows only
	the $O(2,10)$ factor to diverge. The arithmetic classification of the
	corresponding boundary strata, together with
	Theorem~\ref{thm-quotient-type}, then determines the possible metric limits.
	
	The present work is closely related to the compactification theory for K3
	surfaces developed in
	\cite{ouyang2025compactificationmetricmodulispace}. The K3 theory supplies
	the underlying Gromov--Hausdorff limits of the universal covers and the
	ambient Satake compactification. The new issue for Enriques surfaces is to
	retain the covering involution throughout the degeneration. This introduces
	both an arithmetic refinement---the classification of the $32$ boundary
	orbits compatible with the Enriques lattice decomposition---and an
	equivariant geometric refinement---the determination of the limiting
	involution and its quotient on every boundary stratum. Together these yield
	the complete metric compactification described above.
	
	The paper is organized as follows.
	Section~2 recalls the period description of Enriques surfaces, the invariant
	and anti-invariant lattices, and the Ricci-flat metric realization.
	Section~3 constructs the partial metric compactification and proves
	continuity of the geometric realization across the non-collapsing locus.
	Section~4 constructs the Satake compactification and carries out the
	arithmetic classification of its $32$ boundary components.
	Section~5 studies the limiting $\IZ_2$-actions on the corresponding K3
	Gromov--Hausdorff limits and determines the quotient geometry on every
	boundary component.
	
	\medskip
	\noindent
	{\bf Acknowledgements.}
	The author gratefully acknowledges Professor Gang Tian for his guidance and
	support. The author is supported by the National Key R\&D Program of China
	No.~2023YFA1009900.

	\section{Preliminaries}
	
	\subsection{Basic Properties of Enriques surfaces}
	An \emph{Enriques surface} $Y$ is a minimal smooth projective complex surface with
	\[ q(Y) = 0,  \quad 2K_Y \sim 0 \quad (K_Y \not\sim 0). \]
	We briefly list some of its fundamental properties here; see \cite{barth-cptcplxsurface-04,enriques-book} for a detailed treatment.
	
	All Enriques surfaces are deformation equivalent. For any Enriques surface $Y$, one has $\pi_1(Y) \cong \mathbb{Z}/2\mathbb{Z}$, and the universal cover is a K3 surface.
	Conversely, if a K3 surface $X$ admits a fixed-point-free holomorphic involution $\sigma$, then $X / \langle \sigma \rangle$ is an Enriques surface \cite[p.~358]{barth-cptcplxsurface-04}.
	Moreover, $\sigma$ acts anti-invariantly on the nowhere-vanishing holomorphic 2-form $\omega_X$:
	\[ \sigma^*\omega_X = -\omega_X.
	\]
	
	The integral cohomology of $Y$ has 2-torsion:
	\[ H^2(Y, \mathbb{Z}) \cong \mathbb{Z}^{10} \oplus \mathbb{Z}/2\mathbb{Z}.
	\]
	The torsion-free part $H^2(Y,\mathbb{Z})_f$, with the intersection pairing, is the even unimodular lattice of rank $10$ and signature $(1,9)$:
	\[ H^2(Y,\mathbb{Z})_f\cong U \oplus E_8(-1).
	\]
	
	Let $\sigma $ be the deck involution on the K3 cover $X$, and let
	$\Lambda\cong U^{\oplus3}\oplus E_8(-1)^{\oplus2}$ be the K3 lattice.
	The covering involution on $\Lambda$ is given by 
	\[
	\rho(z_1\oplus z_2\oplus z_3\oplus x\oplus y)
	= (-z_1)\oplus z_3\oplus z_2\oplus y\oplus x.
	\]
	More precisely, there exists a marking $\phi:H^2(X,\mathbb Z)\to\Lambda$ such that $\phi\circ\sigma^*=\rho\circ\phi$ \cite[Chapter VIII, Lemma 19.1]{barth-cptcplxsurface-04}.
	
	The covering involution induces the orthogonal decomposition $\Lambda_\IR=\Lambda_\IR^+\oplus\Lambda_\IR^-$, where
	\begin{itemize}
		\item $\Lambda^+\cong U(2)\oplus E_8(-2)$ has signature $(1,9)$;
		\item $\Lambda^-=(\Lambda^+)^\perp\cong U\oplus U(2)\oplus E_8(-2)$ has signature $(2,10)$.
	\end{itemize}
	
	\subsection{K\"ahler--Einstein Metrics on Enriques Surfaces}
	
	Set $\Lambda^-_{\mathbb C}=\Lambda^-\otimes_{\mathbb Z}\mathbb C$ and $\Lambda^-_{\mathbb R}=\Lambda^-\otimes_{\mathbb Z}\mathbb R$.
	Since $\sigma^*\omega_X=-\omega_X$, the period of the K3 cover $\phi([\omega_X])$ lies in
	\begin{equation*}
		\Omega_{\mathrm{En}} = \left\{ [\omega] \in \mathbb{P}(\Lambda^- _\mathbb{C}) \mid \langle \omega, \omega \rangle = 0, \, \langle \omega, \bar{\omega} \rangle > 0 \right\}.
	\end{equation*}
	This domain parametrizes oriented positive-definite $2$-planes in $\Lambda^-_{\mathbb R}$ and is isomorphic to $O(2,10)/(O(2)\times O(10))$.
	Define
	\begin{equation*}
		K\Omega_{\mathrm{En}} = \left\{ (\kappa, [\omega]) \mid \kappa \in \Lambda^+_\mathbb{R} \text{ with } (\kappa,\kappa)=1, \, [\omega] \in \Omega_{\mathrm{En}} \right\}.
	\end{equation*}
	
	Define $\Delta$ to be the set of $(-2)$-roots in $\Lambda$.
	The smooth locus is obtained by removing the hypersurfaces orthogonal to $(-2)$-roots:
	\begin{equation*}
		K\Omega_{\mathrm{En}}^0 =K\Omega_{\mathrm{En}} \setminus (H^-\cup H).
	\end{equation*}
	Here $H^-=\bigcup_{[\delta]\in \Delta \cap \Lambda^-}\{(\kappa,[\omega])\mid [\omega]\in\delta^\perp\}$ and $H=\bigcup_{\delta\in \Delta}\{(\kappa,[\omega])\mid [\omega],\kappa\in\delta^\perp\}$.
	
	\begin{definition}
		A \emph{marked K\"ahler Enriques surface} is a triple $(Y,\phi,\kappa)$, where $Y$ is an Enriques surface, $\kappa$ is a K\"ahler class on $Y$, and $\phi:H^2(X,\mathbb Z)\to\Lambda$ is a marking satisfying
		$\phi\circ\sigma^*=\rho\circ\phi$, where $\sigma$ is the deck involution on the K3 cover $X$.
	\end{definition}

	\begin{lemma}
		For every $(\kappa,[\omega])\in K\Omega_{\mathrm{En}}^0$, there is a unique (unit volume) marked Enriques pair  $(Y,\phi,\kappa_Y)$ with $p:X\to Y$ the K3 cover, such that $\phi(p^*\kappa_Y)=\kappa$ and $ \phi([\omega_X])=[\omega]$ for a holomorphic volume form $\omega_X$.
	\end{lemma}

	\begin{proof}
		
		Let $(\kappa,[\omega])\in K\Omega_{\mathrm{En}}^0$. By surjectivity of the K3 period map \cite[Ch. VIII, Thm. 14.1]{barth-cptcplxsurface-04}, there is a unique marked K3 surface $(X,\phi)$ with period $[\omega]$ and K\"ahler class $\phi^{-1}(\kappa)$.
		The composition $j = \phi^{-1} \circ \rho \circ \phi$ defines a Hodge isometry of $H^2(X, \mathbb{Z})$.
		Because $\kappa\in\Lambda^+_{\mathbb R}$, $j$ preserves the K\"ahler class and is effective. Strong Torelli \cite[Ch. VIII, \S 11]{barth-cptcplxsurface-04} therefore realizes $j$ by a unique involution $\sigma:X\to X$. The condition $(\kappa,[\omega])\notin H^-$ makes $\sigma$ fixed-point-free \cite[Thm. 21.4]{barth-cptcplxsurface-04}, so $Y=X/\langle\sigma\rangle$ is an Enriques surface.
		Since $\kappa$ is $\sigma$-invariant, it descends to a K\"ahler class $\kappa_Y$ on $Y$ satisfying $\phi(p^*\kappa_Y)=\kappa$.
	\end{proof}

	Let $
	G\coloneqq O(1,9)\times O(2,10) $ and let $
	K\coloneqq O(9)\times O(2)\times O(10)$ be its maximal compact subgroup,
	and let
	\[
	G_\IZ\coloneqq
	\operatorname{Aut}(\Lambda)\cap G
	=
	\left\{g\in\operatorname{Aut}(\Lambda)\mid g\rho=\rho g\right\}.
	\]
	Forgetting the marking and taking the quotient gives
	\[ \mathcal{M}_{\mathrm{En}} \coloneqq G_\IZ \backslash K\Omega_{\mathrm{En}}
	\cong G_\IZ \backslash G/K,
	\]
	and we define $\mathcal{M}_{\mathrm{En}}^0 \coloneqq G_\IZ \backslash K\Omega_{\mathrm{En}}^0$. Combined with the Calabi--Yau theorem \cite{yau-78}, the lemma above defines a surjective geometric realization map after rescaling:
	\[ \Phi^0 : \mathcal{M}_{\mathrm{En}}^0 \to \mathfrak{M}_{\mathrm{En}}^{sm}, \]
	where $\mathfrak{M}_{\mathrm{En}}^{sm}$ denotes the set of isometry classes of all smooth Ricci-flat metrics on Enriques surfaces of unit diameter.
	
	\begin{remark}
		From the construction, $\Phi^0$ is surjective. It is in fact bijective: For a smooth hyperk\"ahler metric on the K3 cover $X$, the $3$-plane spanned by the hyperk\"ahler triple is exactly the space of self-dual $2$-forms. Its $\sigma$-invariant line gives the K\"ahler form, and the $(-1)$-eigenspace gives $[\omega]$.
		
	\end{remark}
	\begin{remark}
		By a result of Hitchin \cite[\S 2]{hitchin-74}, every scalar-flat K\"ahler metric on a compact complex surface with torsion canonical bundle is Ricci-flat. Thus the construction also classifies scalar-flat K\"ahler metrics on Enriques surfaces.
	\end{remark}
	
	\section{Partial Compactification of Moduli space}
	
	The aim of this section is to prove the following extension theorem.
	\begin{theorem}\label{thm:continuous-non-collapsing}
		The map $\Phi^0$ extends continuously to a surjection
		\[ \Phi : \mathcal{M}_{\mathrm{En}} \to \mathfrak{M}_{\mathrm{En}}, \]
		where $\FM\en$ is obtained from $\FM\en^{sm}$ by adding all its non-collapsing limits.
	\end{theorem}
	
	Recall that for the K3 period domain $\CM_{K3}\cong \Gamma\backslash O(3,19)/(O(3)\times O(19))$, there is a continuous map $\Phi_{K3}: \CM_{K3}\to \FM_{K3}$ (see \cite{anderson-92} and \cite[Proposition 6.7]{odaka-oshima-21}), and there is a natural map $f: \CM\en \to\CM_{K3}$. From the construction, we see that $\Phi^0(\xi)$ is a $\IZ_2$-quotient of $\Phi_{K3}(f(\xi))$. For $\xi_\infty\in \CM\en\setminus \CM\en^0$, we define $\Phi(\xi_\infty)$ as a $\IZ_2$-quotient of $\Phi_{K3}(f(\xi_\infty))$.
	
	To prove Theorem~\ref{thm:continuous-non-collapsing}, we first establish some lemmas.
	The first lemma is on the topology of the period domain:
	\begin{lemma} \label{lem:connected_nbhd}
		Every $x\in\mathcal{M}_{\mathrm{En}}\setminus\mathcal{M}_{\mathrm{En}}^0$ has a neighborhood basis $\{U_j\}$ such that $U_j\cap\mathcal{M}_{\mathrm{En}}^0$ is connected.
	\end{lemma}
	
	\begin{proof}
		\begin{equation*}
			K\Omega\en \setminus K\Omega\en^{\circ} = H^-\cup H.
		\end{equation*}
		The right-hand side is a locally finite union of hyperplanes of real codimension at least $2$ (local finiteness follows from \cite[Ch.~VIII, Rem.~2.2]{huybrechts-lecturenotes} by restriction from the period domain of K3 surfaces). Removing such a union does not disconnect a sufficiently small connected neighborhood,
		so the image of such a neighborhood under $K\Omega\en\to\CM\en$ satisfies the condition in the lemma.
	\end{proof}

	The second lemma concerns properties of limit spaces:
	\begin{lemma} \label{lem:discrete_topology}
		Let $X$ be a compact metric space. Define the set $\mathcal{A}_X$ to be the collection of isometry classes of $\mathbb Z_2$-quotients of $X$.
		If $\operatorname{Isom}(X)$ is a Lie group, then $\mathcal{A}_X$ is finite. In particular, $\mathcal{A}_X$ is endowed with the discrete topology.
	\end{lemma}
	
	\begin{proof}
		Set $G=\operatorname{Isom}(X)$ and $\mathcal I=\{g\in G\mid g^2=e\}$.
		For any $g \in \mathcal{I}$, let $O_g \coloneq \{hgh^{-1} \mid h \in G\}$ be its conjugacy class.
		We know that $G$ is compact by the Arzel\`a--Ascoli theorem; hence $O_g$ is compact and closed in $\mathcal I$. We show that it is also open.
		
		Suppose $g\exp(Y)\in\mathcal I$ for small $Y\in\mathfrak g$.
		Since $g^2 = e$, the adjoint action $\mathrm{Ad}_g$ is an involution on $\mathfrak{g}$, yielding a decomposition $\mathfrak{g} = \mathfrak{g}_+ \oplus \mathfrak{g}_-$ corresponding to the $\pm 1$ eigenspaces.
		Write $\exp(Y)=\exp(Y_+)\exp(Y_-)$ accordingly. Then $g\exp(Y_+)=\exp(Y_+)g$ and
		similarly, $g \exp(Y_-) = \exp(-Y_-) g$.
		Because $g \exp(Y)$ is an involution, we have:
		\begin{align*}
			\mathrm{Id}=(g \exp(Y))^2 &= g \exp(Y_+) \exp(Y_-) g \exp(Y_+) \exp(Y_-) \\
			&= \exp(Y_+) \exp(-Y_-) \exp(Y_+) \exp(Y_-).
		\end{align*}
		This implies $\exp(-Y_+)=\exp(\operatorname{Ad}_{\exp(Y_-)}Y_+)$. Local injectivity of the exponential map gives
		\begin{equation*}
			(\mathrm{Id} + \mathrm{Ad}_{\exp(Y_-)}) Y_+ = 0.
		\end{equation*}
		For small $Y_-$ this operator is invertible, so $Y_+=0$. Therefore
		\[ g \exp(Y) = g \exp(Y_-) = \exp(-\frac{1}{2}Y_-) g \exp(\frac{1}{2}Y_-), \]
		so $g\exp(Y)$ is conjugate to $g$. Thus every conjugacy class in $\mathcal I$ is open.
		
		Compactness of $\mathcal I$ then implies that it has finitely many conjugacy classes; this proves the finiteness of $\mathcal{A}_X$.
	\end{proof}
	
	By \cite[Thm.~1.21]{colding-naber-12}, the isometry group of a Ricci limit space is a Lie group. 
	In particular, the result holds for any GH limit of hyperk\"ahler K3 surfaces.

	Next, we relate Enriques limits to their K3 covers. By the Arzel\`a--Ascoli theorem, we have:
	
	\begin{lemma}\label{lem:limit_quotient} 
		Consider a (collapsing or non-collapsing) sequence of KE metrics on Enriques surfaces $(Y_i,h_i)\to (Y_\infty,h_\infty)$.
		Let $(X_i,g_i)$ be the corresponding KE metrics on K3 surfaces. Suppose $(X_i,g_i)$ converges to $(X_\infty,g_\infty)$.
		Then there is an isometric $\IZ_2$-action on $X_\infty$ with $(Y_\infty,h_\infty)=(X_\infty,g_\infty)/\IZ_2$.
	\end{lemma} 
		%
		%

	\begin{proof}[Proof of Theorem~\ref{thm:continuous-non-collapsing}]
		
		Let $\xi_\infty\in\mathcal{M}_{\mathrm{En}}$. Define $\mathfrak{F}(\xi_\infty)$ to be the collection of all Gromov--Hausdorff limits of sequences $\Phi^0(\xi_i)$ such that $\xi_i\to \xi_\infty$. By Lemma~\ref{lem:limit_quotient}, every element of $\mathfrak{F}(\xi_\infty)$ is an isometric $\mathbb{Z}_2$-quotient of the corresponding K3 limit $X_\infty$. Hence $
		\mathfrak{F}(\xi_\infty)\subseteq \mathcal{A}_{X_\infty}$.
		As in \cite[Prop.~5.9]{ouyang2025compactificationmetricmodulispace}, Lemma~\ref{lem:connected_nbhd} implies that $\mathfrak F(\xi_\infty)$ is connected. Lemma~\ref{lem:discrete_topology} then makes it a singleton.
		Define $\Phi(\xi_\infty)$ to be this unique limit. Thus $\Phi^0$ extends continuously to a map
		\[
		\Phi:\mathcal{M}_{\mathrm{En}}\longrightarrow\mathfrak{M}_{\mathrm{En}}.
		\]
		
		Surjectivity follows from \cite[Theorem II]{anderson-92}: if a sequence diverges in $\CM\en$, its K3 covers collapse, and so do the Enriques quotients. Thus all the non-collapsing limits lie in the image of $\Phi$. This proves the theorem.
	\end{proof}
	\begin{remark}
		If orientation data are retained, then $\Phi$ is bijective. Indeed, let $(Y_\infty,g_\infty)\in \FM\en$ be an oriented non-collapsing limit. Then $Y_\infty$ is a $\IZ_2$-quotient of an oriented hyperk\"ahler K3 orbifold $X_\infty$. Let $X_\infty^0$ and $Y_\infty^0$ denote the regular parts of $X_\infty$ and $Y_\infty$, respectively, each obtained by removing finitely many points. The metric and orientation determine the space of parallel self-dual $2$-forms on the universal cover of $Y_\infty^0$, and this space is spanned by a hyperk\"ahler triple. The orbifold $X_\infty$ is determined by the property that this hyperk\"ahler triple descends to $X_\infty^0$. Thus $X_\infty$ and the Enriques involution are determined by $(Y_\infty,g_\infty)$ together with its orientation; therefore, the period point in $\CM\en$ is also determined.
		
	\end{remark}

	\section{Satake Compactification}

	\subsection{Structure of $\overline{G/K}^{\mathrm{Sat}}$}
	Our compactification is a slightly generalized version of the Satake compactification in \cite{satake-60-1}; we also refer to \cite{odaka-oshima-21,ouyang2025compactificationmetricmodulispace} for the Satake compactification of the K3 moduli space.
	Recall that
	\[
	G=O(1,9)\times O(2,10),
	\qquad
	K=O(1)\times O(9)\times O(2)\times O(10).
	\]
	We also set
	\[
	G_{K3}\coloneqq O(3,19),
	\qquad
	K_{K3}\coloneqq O(3)\times O(19).
	\]
	Thus $G_{K3}/K_{K3}$ is the symmetric space associated with the K3 lattice of signature $(3,19)$.
	The main result of this subsection, Theorem~\ref{thm:boundary_components}, describes the structure of $\overline{G/K}^{\mathrm{Sat}}$ and states that the compactification is obtained by adding components of the form $G_A/K_A$, where $A$ is a $\rho$-open set (Definition~\ref{definition-rho-open}).

	Let the Lie algebra be given by $\mathfrak{g} = \mathfrak{g}_1 \oplus \mathfrak{g}_2 = \mathfrak{so}(1,9) \oplus \mathfrak{so}(2,10)$.
	Let 
	\begin{equation} \label{equation-V-g-K3}
		V = \mathfrak{g}_{K3} = \mathfrak{so}(3,19)= \begin{pmatrix} \mathfrak{g}_1 & \star \\ \star & \mathfrak{g}_2 \end{pmatrix}
	\end{equation} 
	be the standard block matrix realization.
	We consider the representation $\rho: G \to SL(V, \mathbb{C})$.
	
	\begin{definition}
		There is a sequence of embeddings
		\[
		G/K \hookrightarrow G_{K3}/K_{K3} \hookrightarrow SL(V, \mathbb{C})/SU(V) \simeq \mathcal{P}(V) \hookrightarrow \mathbb{P}(\mathcal{H}(V)).
		\]
		The Satake compactification $\overline{G/K}^{\mathrm{Sat}}$ is defined as the closure of $G/K$ in $\mathbb{P}(\mathcal{H}(V))$.
		Here $\mathcal{P}(V)$ denotes the space of positive-definite Hermitian matrices of determinant $1$, and $\mathbb{P}(\mathcal{H}(V))$ denotes the projectivization of the space $\mathcal{H}(V)$ of Hermitian matrices on $V$.
	\end{definition}
	
	Equivalently, the compactification of $G/K$ can also be defined as its closure in $\overline{G_{K3}/K_{K3}}^{\mathrm{Sat,ad}}$.

	We now fix some Lie-algebra notation. Recall that the K3 lattice decomposes as $\Lambda_\IR=\Lambda_\IR^+\oplus\Lambda_\IR^-$ according to the Enriques involution. Consider the standard representations
	\[
	\mathfrak g_1\longrightarrow\mathfrak{sl}(\Lambda^+_{\mathbb C})
	\cong\mathfrak{sl}(\mathbb C^{10}),
	\qquad
	\mathfrak g_2\longrightarrow\mathfrak{sl}(\Lambda^-_{\mathbb C})
	\cong\mathfrak{sl}(\mathbb C^{12}).
	\]
	Denote their weights by $\{\pm e_i\}_{i=1}^5$ and $\{\pm f_j\}_{j=1}^6$, respectively. We choose the corresponding systems of simple roots
	\[
	\Pi_1=\{e_1-e_2,\ldots,e_4-e_5,e_4+e_5\},
	\qquad
	\Pi_2=\{f_1-f_2,\ldots,f_5-f_6,f_5+f_6\}.
	\]
	
	Let $\mathfrak g=\mathfrak k\oplus\mathfrak p$ be the Cartan decomposition associated with $K$, and choose a Cartan subalgebra $\mathfrak h$ such that
	\[
	\mathfrak h^- \coloneqq \mathfrak h\cap\mathfrak p
	=\mathfrak h_1^-\oplus\mathfrak h_2^-
	\]
	is maximal abelian in $\mathfrak p$. After identifying the Cartan subalgebra with its dual via the invariant bilinear form, we may choose the above coordinates so that
	\[
	\mathfrak h_1^-=\langle e_1\rangle,
	\qquad
	\mathfrak h_2^-=\langle f_1,f_2\rangle.
	\]
	
	Let $W$ denote the set of weights of the representation of $\mathfrak g$ on $V$. Explicitly,
	\[
	W=\{\pm a\pm b\mid a,b\in\{e_i,f_j\}\}.
	\]
	As a $\mathfrak g$-module, $V$ decomposes as $V=V_1\oplus V_2\oplus V_3$, with highest weights
	\[
	\lambda_1=e_1+e_2,\qquad
	\lambda_2=e_1+f_1,\qquad
	\lambda_3=f_1+f_2.
	\]
	Set
	\[
	W_{\mathrm{hw}}\coloneqq\{\lambda_1,\lambda_2,\lambda_3\},
	\]
	the set of highest weights of the irreducible summands of $V$.
	Here $V_1=\mathfrak g_1$, $V_3=\mathfrak g_2$, and $V_2$ is the off-diagonal block in Equation~\ref{equation-V-g-K3}.

	The closed positive Weyl chamber is
	\[
	\begin{aligned}
		\mathfrak a^-
		&\coloneqq
		\{H\in\mathfrak h^-:\alpha(H)\geq 0
		\quad\text{for every }\alpha\in\Pi_1\cup\Pi_2\}\\
		&=
		\{xe_1+yf_1+zf_2:x\geq 0,\ y\geq z\geq 0\}.
	\end{aligned}
	\]
	We next define the subsets of $W$ that determine the boundary strata. Intuitively, a $\rho$-open set consists of the weights that have comparable maximal growth along a sequence in $\mathfrak a^-$.
	
	\begin{definition}\label{definition-rho-open}
		A nonempty subset $A\subseteq W$ is called \emph{$\rho$-open} if there exists a sequence $H_i\in\mathfrak a^-$ such that, for all $\alpha_1,\alpha_2\in A$ and $\alpha_3\in W\setminus A$, the sequence $(\alpha_1-\alpha_2)(H_i)$ remains bounded, whereas
		\[
		(\alpha_1-\alpha_3)(H_i)\longrightarrow+\infty
		\qquad\text{as }i\longrightarrow\infty.
		\]
	\end{definition}
	
	\begin{definition}
		Set
		\[
		\Pi\coloneqq\Pi_1\cup\Pi_2
		\cup\{\pm(\lambda_i-\lambda_j):1\leq i<j\leq 3\}.
		\]
		For a $\rho$-open set $A\subseteq W$, define
		\[
		\Pi_A\coloneqq
		\{\beta\in\Pi:\beta=\alpha-\alpha'
		\text{ for some }\alpha,\alpha'\in A\}.
		\]
		
		Let $R_A$ be the root system spanned by $\Pi_A\cap (\Pi_1\cup\Pi_2)$, and set
		\[
		V_A\coloneqq\bigoplus_{\alpha\in A}V_\alpha,
		\mathfrak g_{A,\mathbb C}
		\coloneqq
		\operatorname{span}_{\mathbb C}(\Pi_A)
		+\sum_{\alpha\in R_A}\mathfrak g_\alpha
		\subseteq\mathfrak g_{\mathbb C},
		\qquad
		\mathfrak g_A\coloneqq
		\mathfrak g\cap\mathfrak g_{A,\mathbb C}.
		\]
		Let $G_A$ be the connected subgroup of $G$ with Lie algebra $\mathfrak g_A$, and let $K_A\coloneqq G_A\cap K$. The restriction of the representation to $\mathfrak g_A$ induces a map, again denoted by $\rho_A$,
		\[
		\rho_A:G_A/K_A
		\longrightarrow SL(V_A)/SU(V_A)
		\hookrightarrow\mathbb P(\mathcal H(V_A))
		\hookrightarrow\mathbb P(\mathcal H(V)),
		\]
		where $\mathbb{P}(\mathcal{H}(V_A)) \hookrightarrow \mathbb{P}(\mathcal{H}(V))$ is defined by assigning $0$ to the orthogonal complement of $V_A$.
	\end{definition}

We now classify the $\rho$-open sets and their corresponding components $(G_A,K_A)$.
By definition, $A$ is uniquely determined by $A\cap W_{\mathrm{hw}}$ and
\[
\Pi_A^-\coloneqq\operatorname{proj}_{\mathfrak h^-}(\Pi_A).
\]
The $\rho$-open sets $A\neq W$ are listed in Table~\ref{table:rho-open} (for $A=W$, we have $(G_A,K_A)=(G,K)$). We label each $\rho$-open set $A$ consistently with the corresponding K3 boundary type.
%

\begin{table}[htbp]
	\caption{The $\rho$-open sets $A\neq W$ and their boundary data.}
	\label{table:rho-open}
	\centering
	\renewcommand{\arraystretch}{1.35}
	\resizebox{\textwidth}{!}{%
		\begin{tabular}{c|c|c|c}
			\hline
			Type & $A\cap W_{\mathrm{hw}}$ & $\Pi_A^-$ & $(G_A,K_A)$ \\
			\hline
			$a_2$ & $\{\lambda_2,\lambda_3\}$
			& $\{e_1,f_2,e_1-f_2\}$
			& $(SO(1,9)\times SO(1,9),SO(9)\times SO(9))$ \\
			$b$ & $\{\lambda_2,\lambda_3\}$
			& $\{f_1-f_2,e_1-f_2\}$
			& $(\mathbb R\times SL_2(\mathbb R),SO(2))$ \\
			$c_2$ & $\{\lambda_2,\lambda_3\}$
			& $\{e_1-f_2\}$
			& $(\mathbb R,\{0\})$ \\
			$a_1$ & $\{\lambda_1,\lambda_2\}$
			& $\{f_1-f_2,f_2\}$
			& $(SO(9)\times SO(2,10),SO(9)\times SO(2)\times SO(10))$ \\
			$c_1$ & $\{\lambda_2\}$
			& $\{f_1-f_2\}$
			& $(SL_2(\mathbb R),SO(2))$ \\
			$d_1$ & $\{\lambda_2\}$
			& $\emptyset$
			& $(\{0\},\{0\})$ \\
			$d_2$ & $\{\lambda_3\}$
			& $\emptyset$
			& $(\{0\},\{0\})$ \\
			\hline
		\end{tabular}%
	}
\end{table}

Now we recall the boundary structure of the adjoint Satake compactification of $G_{K3}/K_{K3}$. Let $\rho_{K3}=\operatorname{Ad}:G_{K3}\to SL(V,\mathbb C)$. 
There are $4$ kinds of boundaries
\[
\rho_{K3,t}:G_{K3,t}/K_{K3,t}
\hookrightarrow\mathbb P(\mathcal H(V_t))
\hookrightarrow\mathbb P(\mathcal H(V)), \text{for } t\in\{a,b,c,d\},
\]
(the choice of $\rho_{K3,t}$ depends on a choice of simple roots of $\mathfrak{g}_{K3}$),
and the compactification satisfies
\[
\partial\overline{G_{K3}/K_{K3}}^{\mathrm{Sat,ad}}
=
\bigcup_{t\in\{a,b,c,d\},k\in K_{K3}}
k\cdot\rho_{K3,t}(G_{K3,t}/K_{K3,t}),
\]
where
\[
G_{K3,t}/K_{K3,t}
\cong
\begin{cases}
	SO(2,18)/(SO(2)\times SO(18)), & t=a,\\
	SL_3(\mathbb R)/SO(3), & t=b,\\
	SL_2(\mathbb R)/SO(2), & t=c,\\
	\mathrm{pt}, & t=d.
\end{cases}
\]

For $A$ in Table~\ref{table:rho-open}, let $t(A)$ be the corresponding ambient K3 type.
By the weight classification, there is $k_A\in K_{K3}$ such that
\[
\begin{aligned}
	V_A&=\rho_{K3}(k_A)V_{t(A)},\\
	G_A&=G\cap k_A G_{K3,t(A)},\\
	K_A&=G\cap k_A K_{K3,t(A)}.
\end{aligned}
\]
Moreover, $\rho_A$ is also given by the composite
\[
G_A/K_A
\hookrightarrow
k_A\cdot G_{K3,t(A)}/K_{K3,t(A)}
\hookrightarrow
\overline{G_{K3}/K_{K3}}^{\mathrm{Sat,ad}}
\hookrightarrow
\mathbb P(\mathcal H(V)).
\]

\begin{theorem}\label{thm:boundary_components}
	\begin{enumerate}
		\item The Satake compactification has the stratification
		\[
		\overline{G/K}^{\mathrm{Sat}}
		=
		\bigcup_{\substack{\rho\text{-open set} \ A}}
		K\cdot\rho_A(G_A/K_A).
		\]
		\item Any $2$ strata $k\cdot\rho_A(G_A/K_A)$ and $k'\cdot\rho_{A'}(G_{A'}/K_{A'})$ are either equal or disjoint.
	\end{enumerate}
\end{theorem}

\begin{proof}
	\textbf{(1)} We follow Satake's argument \cite[\S4.1]{satake-60-1}. By the Cartan decomposition \cite[Thm.~7.39]{knapp-lie-group-beyond},
	\[
	G=K\exp(\mathfrak a^-)K.
	\]
	It therefore suffices to study sequences represented by $\rho(\exp H_i)$ with $H_i\in\mathfrak a^-$. After passing to a subsequence, let $A\subseteq W$ be the $\rho$-open set of weights of maximal growth as in Definition~\ref{definition-rho-open}. Then $(\alpha_1-\alpha_2)(H_i)$ remains bounded for $\alpha_1,\alpha_2\in A$, whereas
	\[
	(\alpha_1-\alpha_3)(H_i)\longrightarrow+\infty
	\qquad
	(\alpha_3\in W\setminus A).
	\]
	After passing to a further subsequence, $\gamma(H_i)$ converges for every $\gamma\in\Pi_A$, and hence for every $\gamma\in\Pi_A^-$. Recall that
	\[
	\Pi_A^-\coloneqq\operatorname{proj}_{\mathfrak h^-}(\Pi_A).
	\]
	The limits $c_\gamma\coloneqq\lim_{i\to\infty}\gamma(H_i)$, for $\gamma\in\Pi_A^-$, define a linear functional on
	\[
	\operatorname{span}_{\mathbb R}(\Pi_A^-)
	\subseteq\mathfrak h^-.
	\]
	Since the restriction of the Killing form to $\mathfrak h^-$ is positive definite, it identifies this subspace with its dual. Hence there is a unique $
	H_{A,\infty}\in\operatorname{span}_{\mathbb R}(\Pi_A^-)$
	such that
	\[
	\gamma(H_{A,\infty})=c_\gamma
	\qquad\text{for every }\gamma\in\Pi_A^-.
	\]
	In particular, $\gamma(H_i-H_{A,\infty})\to0$ for every $\gamma\in\Pi_A^-$. Since $H_i,H_{A,\infty}\in\mathfrak h^-$, this also holds for every $\gamma\in\Pi_A$.
	
	Since $A$ represents those weights with maximal growth, the limit of $
	\rho(\exp H_i)$ lies in $\mathbb P(\mathcal H(V_A))$. Moreover, from the construction we obtain
	\[
	\lim_{i\to\infty}\rho(\exp H_i)
	=
	\rho_A(\exp H_{A,\infty})
	\in\mathbb P(\mathcal H(V_A)).
	\]
	Consequently,
	\[
	\overline{G/K}^{\mathrm{Sat}}
	\subseteq
	\bigcup_{A,H_{A,\infty}}K\cdot\rho_A(\exp H_{A,\infty})
	\subseteq 
	\bigcup_A K\cdot\rho_A(G_A/K_A).
	\]
	The reverse inclusion follows from the construction of each boundary stratum, proving the equality in (1).
	
	\textbf{(2)} For the K3 Satake compactification, any two boundary strata are either equal or disjoint; see \cite[\S4]{satake-60-1} or \cite[Cor.~I.4.32]{borel-ji-06}.
	
	Let $S_E$ denote the image of $G/K$ in $G_{K3}/K_{K3}$. By the compatibility established above, every Enriques stratum is of the form
	\[
	k\cdot\rho_A(G_A/K_A)
	=
	\overline{S_E}\cap
	kk_A\cdot\rho_{K3,t(A)}
	\bigl(G_{K3,t(A)}/K_{K3,t(A)}\bigr).
	\]
	If two Enriques strata intersect, then their corresponding K3 strata intersect and hence coincide. Intersecting this common K3 stratum with $\overline{S_E}$ shows that the $2$ Enriques strata are equal, proving (2).
\end{proof}

\subsection{Compactification of $\CM\en $ and Extendability of $\Phi$}

We now pass from the compactification of $G/K$ to its arithmetic quotient. Recall that
\[
G_\IQ \coloneqq \mathrm{Aut}(\Lambda_\IQ) \cap G,
\qquad
G_\IZ \coloneqq \mathrm{Aut}(\Lambda) \cap G,
\qquad
\CM\en=G_\IZ\backslash G/K.
\]
We collect all those rational boundaries:
\[
S^{\mathbb Q}
\coloneqq
\bigcup_{A: \rho \text{-open}}
G_\IQ\cdot\rho_A(G_A/K_A)
\subseteq \overline{G/K}^{\mathrm{Sat}}.
\]
Then the Satake compactification of the arithmetic quotient is given by the following quotient \cite[\S\S3.3--3.4]{satake-60-2}:
\[
\overline{\CM\en}^{\mathrm{Sat}}
\coloneqq
G_\IZ\backslash S^{\mathbb Q}.
\]
The compactification is obtained by adding finitely many boundary components of the form $G_{A,\IZ}\backslash G_A/K_A$, where $G_{A,\IZ}\coloneq G_A\cap G_\IZ$.
Moreover, the Satake topology restricts to the standard topology on each Siegel set in $S^{\mathbb Q}$. The proof of \cite[Prop.~3.20]{ouyang2025compactificationmetricmodulispace} applies verbatim to give:

\begin{proposition}\label{proposition-connectness-satake2}
	For every $y\in\partial\overline{\CM\en}^{\mathrm{Sat}}$, there is a neighborhood basis $\{U_r\}$ of $y$ such that $U_r\cap\CM\en$ is connected.
\end{proposition}

By the same method as in \cite[Prop.~5.9]{ouyang2025compactificationmetricmodulispace},  
Proposition~\ref{proposition-connectness-satake2} and Lemmas~\ref{lem:discrete_topology} and~\ref{lem:limit_quotient} imply that $\Phi$ extends continuously to
\[
\overline{\Phi}:
\overline{\CM\en}^{\mathrm{Sat}}
\longrightarrow
\overline{\FM\en}^{\mathrm{GH}}.
\]
It remains to identify $\overline{\Phi}$ on each boundary stratum $\rho_A(G_A/K_A)$.

\subsection{Boundary Components}

By construction, $S^{\mathbb Q}$ contains the $G_\IQ$-translates of
the rational Satake strata, and a stratum $\rho_A(G_A/K_A)$ descends to the
arithmetic boundary component $G_{A,\IZ}\backslash G_A/K_A$. Its stabilizer
in $G$ is a parabolic subgroup, equivalently the stabilizer of an isotropic flag in $\Lambda_{\mathbb Q}$. Witt's theorem
\cite[Ch.~IV, Thm.~3]{serre-73} implies that $G_\IQ$ is transitive on
rational isotropic flags of each fixed type. Hence the boundary components are
classified by the $G_\IZ$-equivalence classes of rational isotropic flags in
$\Lambda_{\mathbb Q}$.

We use the decompositions
\[
\Lambda^+
\cong E_8(-2)\oplus U(2)
= E_8(-2)\oplus\langle e_2^+,f_2^+\rangle
\]
and
\[
\Lambda^-
\cong E_8(-2)\oplus U(2)\oplus U
= E_8(-2)
\oplus\langle e_2^-,f_2^-\rangle
\oplus\langle e_1^-,f_1^-\rangle,
\]
where the basis vectors satisfy
\[
(e_2^+,f_2^+)=(e_2^-,f_2^-)=2,
\qquad
(e_1^-,f_1^-)=1,
\]
and all $6$ basis vectors are isotropic. Moreover, we may choose these bases compatibly with the gluing
$\Lambda^+\oplus\Lambda^-\subset\Lambda$ so that
\[
e_2^++e_2^-,f_2^++f_2^-\in2\Lambda.
\]

The following theorem will be proved in \S~\ref{section-proof-thm-isotropic-class}:
\begin{theorem} \label{thm:isotropic_classes}
	There are exactly $32$ boundary components of
	$\overline{\CM\en}^{\mathrm{Sat}}$, corresponding to
	$32$ $G_\IZ$-equivalence classes of rational isotropic
	flags, distributed as follows.
	
	\begin{enumerate}
		\item[Type $a_1$] $1$-dimensional in $\Lambda^+$:
		\begin{itemize}
			\item Type $a_{1,1}$: $\langle e_2^+\rangle$.
		\end{itemize}
		\item[Type $a_2$] $1$-dimensional in $\Lambda^-$:
		\begin{itemize}
			\item Type $a_{2,1}$: divisor $1$: $\langle e_1^-\rangle$.
			\item Type $a_{2,2}$: divisor $2$: $\langle e_2^-\rangle$.
		\end{itemize}
		\item[Type $b$] $3$-dimensional in $\Lambda^+\oplus\Lambda^-$.
		We first identify the copies of $E_8(-2)$ in $\Lambda^+$ and $\Lambda^-$ so that if $x^+\in\Lambda^+$ and $x^-\in\Lambda^-$ correspond under this identification, then $\frac 12(x^++x^-) \in \Lambda$.
		Choose primitive vectors $x_+\in E_8(-2)\subset\Lambda^+$ and
		$x_-\in E_8(-2)\subset\Lambda^-$ with $x_+^2=-16$ and $x_-^2=-8$
		such that $[x_+/2]_{\Lambda^+}$ is orthogonal to
		$\psi^{-1}([x_-/2]_{\Lambda^-})$ (see \S~\ref{section-lattice} for the definition). The $6$ classes are represented by the following:
		
		\begin{itemize}
			\item Type $b_1$:
			$W_{b_1}=\langle e_2^-,e_1^-,e_2^+\rangle$.
			\item Type $b_2$:
			$W_{b_2}=\langle e_2^-,e_1^-,f_2^+\rangle$.
			\item Type $b_3$:
			$W_{b_3}=\langle e_2^-,e_1^-,2e_2^++2f_2^++x_+\rangle$.
			\item Type $b_4$:
			$W_{b_4}=
			\langle e_2^-,2e_1^-+2f_1^-+x_-,e_2^+\rangle$.
			\item Type $b_5$:
			$W_{b_5}=
			\langle e_2^-,2e_1^-+2f_1^-+x_-,f_2^+\rangle$.
			\item Type $b_6$:
			$W_{b_6}=
			\langle e_2^-,2e_1^-+2f_1^-+x_-,
			2e_2^++2f_2^++x_+\rangle$.
		\end{itemize}
		\item[Type $c_1$] Isotropic flags of dimension type $(1,3)$,
		$I_1\subset I_3$, where
		$I_1\subset\Lambda^+$.
		For $I_3=W_{b_i}$, the flag is
		\[
		(W_{b_i}\cap\Lambda^+)\subset W_{b_i}.
		\]
		For each $1\leq i\leq6$, this gives $1$ class $c_{1,i}$.
		\item[Type $c_2$] Isotropic flags of dimension type $(1,3)$,
		$I_1\subset I_3$, where
		$I_1\subset\Lambda^-$.
		The $11$ classes have the following representatives:
		\begin{itemize}
			\item For $i=1,2,3$,
			\[
			c_{2,ia}:\ \langle e_1^-\rangle
			\subsetneq W_{b_i},
			\qquad
			c_{2,ib}:\ \langle e_2^-\rangle
			\subsetneq W_{b_i}.
			\]
			\item For $i=4,5$,
			\[
			c_{2,ia}:\ \langle e_2^-\rangle
			\subsetneq W_{b_i},
			\qquad
			c_{2,ib}:\ 
			\langle2e_1^-+2f_1^-+x_-\rangle
			\subsetneq W_{b_i}.
			\]
			\item Over $b_6$ there is $1$ class,
			\[
			c_{2,6}:\ \langle e_2^-\rangle
			\subsetneq W_{b_6}.
			\]
		\end{itemize}
		\item[Type $d_1$] $2$-dimensional in $\Lambda^-$:
		\begin{itemize}
			\item Type $d_{1,1}$: $\langle e_2^-,e_1^-\rangle$.
			\item Type $d_{1,2}$:
			$\langle e_2^-,2e_1^-+2f_1^-+x_-\rangle$.
		\end{itemize}
		\item[Type $d_2$] $2$-dimensional in $\Lambda^+\oplus\Lambda^-$:
		\begin{itemize}
			\item Type $d_{2,1}$: $\langle e_2^+,e_2^-\rangle$.
			\item Type $d_{2,2}$: $\langle e_2^+,f_2^-\rangle$.
			\item Type $d_{2,3}$:
			$\langle e_2^+,2e_1^-+2f_1^-+x_-\rangle$.
			\item Type $d_{2,4}$: $\langle e_2^+,e_1^-\rangle$.
		\end{itemize}
	\end{enumerate}
\end{theorem}

\subsection{Lattice preliminaries} \label{section-lattice}

Here we collect some facts used in the orbit calculation. For a
nondegenerate even lattice $M$, set
\[
M^\vee
\coloneqq\{x\in M\otimes\mathbb Q:(x,M)\subseteq\mathbb Z\},
\qquad
A_M\coloneqq M^\vee/M.
\]
The bilinear form on $M$ induces
\[
b_M:A_M\times A_M\longrightarrow\mathbb Q/\mathbb Z,
\qquad
q_M:A_M\longrightarrow\mathbb Q/2\mathbb Z.
\]
We call $\alpha,\beta\in A_M$ orthogonal when $b_M(\alpha,\beta)=0$. For $u\in M^\vee$, let $[u]_M$ denote its
class in $A_M$. If $w\in M$ is primitive, we define $\operatorname{div}_M(w)$ by
\[
\operatorname{div}_M(w)\mathbb Z=(w,M).
\]

Because $\Lambda$ is unimodular and $\Lambda^\pm$ are primitive orthogonal
sublattices, there is a canonical anti-isometry
\[
\psi:(A_{\Lambda^+},q_{\Lambda^+})
\xrightarrow{\sim}
(A_{\Lambda^-},-q_{\Lambda^-}).
\]
Moreover, for $a^\pm\in\Lambda^\pm$ we have
\[
\psi([a^+/2]_{\Lambda^+})=[a^-/2]_{\Lambda^-}
\quad\Longleftrightarrow\quad
\frac{1}{2}(a^++a^-)\in\Lambda.
\]
Further, let $\bar g_\pm\in O(A_{\Lambda^\pm})$ denote the actions induced by
$g_\pm\in O(\Lambda^\pm)$. The
pair $(g_+,g_-)$ extends to an isometry of $\Lambda$ if and only if
\begin{equation} \label{compatible}
	\psi\circ\bar g_+ \circ\psi^{-1}=\bar g_-
\end{equation}
(cf.~\cite[Ch.~14, Prop.~2.6]{huybrechts-lecturenotes}).  Moreover, the
maps $O(\Lambda^\pm)\to O(A_{\Lambda^\pm})$ are surjective because
$\Lambda^\pm$ are indefinite $2$-elementary lattices
(cf.~\cite[Proposition~3.6.3]{nikulin}).

Finally, if $e\in M$ is isotropic and $a\in e^\perp\cap M$, the Eichler
transvection
\[
E_{e,a}(z)
\coloneqq z-(a,z)e+(e,z)a-\frac12(a,a)(e,z)e
\]
lies in $ O(M)$. It fixes $e$ and will be used below to put
isotropic vectors and flags into the representatives listed in
Section~4.3.

\begin{lemma}\label{lem:lattice_properties}
	\begin{enumerate}
		\item The bilinear form on $A_{E_8(-2)}$ is the fourfold orthogonal
		sum $u^{\oplus4}$ of the hyperbolic form on $\mathbb F_2^2$, and
		\[
		O(E_8(-2))\twoheadrightarrow O(A_{E_8(-2)}).
		\]
		\item The stabilizer map
		\[
		O_{e_2^+}(\Lambda^+)
		\twoheadrightarrow
		O_{[e_2^+/2]_{\Lambda^+}}(A_{\Lambda^+})
		\]
		is surjective. Here $O_{e_2^+}(\Lambda^+)$ denotes the stabilizer of $e_2^+$ in $O(\Lambda^+)$, and similarly for $O_{[e_2^+/2]_{\Lambda^+}}(A_{\Lambda^+})$.
		\item If $v_1,v_2\in \Lambda^\pm$ are primitive isotropic vectors with
		\[
		\left[\frac{v_1}{\operatorname{div}(v_1)}\right]_{\Lambda^\pm}
		=
		\left[\frac{v_2}{\operatorname{div}(v_2)}\right]_{\Lambda^\pm},
		\]
		then $v_1,v_2$ are equivalent under
		$\widetilde O(\Lambda^\pm)$. Here
		\[
		\widetilde O(M)
		\coloneqq\ker\bigl(O(M)\longrightarrow O(A_M)\bigr)
		\]
		denotes the stable orthogonal group.
	\end{enumerate}
\end{lemma}

\begin{proof}
	Part~(1) is the standard description of the discriminant form of the
	$E_8$ lattice; see
	\cite[Ch.~VI, \S4, Ex.~1]{bourbaki}. 
	
	For part~(2), let
	$\gamma\in A_{\Lambda^+}$ be an arbitrary isotropic vector with
	$b_{\Lambda^+}([e_2^+/2]_{\Lambda^+},\gamma)=\frac12$. Then we may write 
	$\gamma=[f'/2]_{\Lambda^+}$ with
	\[
	f'=a e_2^++f_2^++x,\qquad x\in E_8(-2).
	\]
	Applying the Eichler transvection $E_{e_2^+,\frac x2}$ reduces this to the case
	$x=0$, and isotropy then gives $a=0$. Part~(1) supplies the remaining
	action on the orthogonal complement, proving surjectivity of the
	stabilizer map. 
	
	For part~(3), the assertion for $\Lambda^+$ follows from (2); the assertion for $\Lambda^-$ is
	\cite[Cor.~3.3]{sterk}.
\end{proof}

\subsection{Proof of Theorem~\ref{thm:isotropic_classes}}\label{section-proof-thm-isotropic-class}

The correspondence between the Satake boundary types and the associated
types of rational isotropic flags is given in 
\cite{ouyang2025compactificationmetricmodulispace}. It therefore
remains to classify, up to $G_\IZ$-equivalence, the rational isotropic flags
listed above. We proceed dimension by
dimension.

\vspace{0.5em}
\noindent \textbf{Cases $a_1$ and $a_2$: $1$-dimensional isotropic
	subspaces $\langle w^\pm\rangle\subset\Lambda^\pm$.}

For $\Lambda^+$, rescale the bilinear form by $1/2$. For every primitive
isotropic vector $w^+\in\Lambda^+(1/2)$, the quotient
$(w^+)^\perp/\langle w^+\rangle$ is isometric to $E_8(-1)$, the unique
negative-definite even unimodular lattice of rank $8$. Hence there is $1$
orbit, of type $a_{1,1}$.

For $\Lambda^-$, \cite[Prop.~4.5]{sterk} gives $2$ orbits, of types
$a_{2,1}$ and $a_{2,2}$, distinguished by
$\operatorname{div}_{\Lambda^-}(w^-)=1$ or $2$.

\vspace{0.5em}
\noindent \textbf{Case $d_1$: $2$-dimensional isotropic planes
	$\langle w_1^-,w_2^-\rangle\subset\Lambda^-$.}
As in Case $a_2$, \cite[Prop.~4.6]{sterk} gives $2$ classes, of types
$d_{1,1}$ and $d_{1,2}$, with divisor types $(1,2)$ and $(2,2)$,
respectively.

\vspace{0.5em}
\noindent \textbf{Case $d_2$: mixed isotropic planes
	$I=\langle w^+,w^-\rangle\subset\Lambda^+\oplus\Lambda^-$.}

By the type-$a_1$ case, we may assume
that $w^+=e_2^+$. Set
\[
\beta\coloneqq[e_2^+/2]_{\Lambda^+},
\qquad
\alpha\coloneqq\psi^{-1}\!\left(
\left[\frac{w^-}{\operatorname{div}(w^-)}\right]_{\Lambda^-}
\right)
\in A_{\Lambda^+}.
\]
The subgroup preserving $e_2^+$ acts on $A_{\Lambda^+}$ through the full
stabilizer $O_\beta(A_{\Lambda^+})$ by
Lemma~\ref{lem:lattice_properties}(2). Every such action can be matched,
via $\psi$, by an isometry of $\Lambda^-$ and hence extends to $\Lambda$.
Once the class $\alpha$ is fixed, Lemma~\ref{lem:lattice_properties}(3)
allows us to choose a representative for $w^-$. Thus the $G_\IZ$-orbit of
$I$ is determined by the $O_\beta(A_{\Lambda^+})$-orbit of $\alpha$.

Suppose first that $\operatorname{div}_{\Lambda^-}(w^-)=2$. Then
$\alpha=\psi^{-1}([w^-/2]_{\Lambda^-})$ is a nonzero isotropic class. Note that
$A_{\Lambda^+}=u^{\oplus5}$ is nondegenerate. Witt's theorem \cite[Ch.~IV, Thm.~3]{serre-73} over $\mathbb F_2$ therefore shows that the
stabilizer of $\beta$ has exactly $3$ orbits on the nonzero isotropic
classes, distinguished by the following:
\begin{itemize}
	\item $\alpha=\beta$:
	$I=\langle e_2^+,e_2^-\rangle$ (type $d_{2,1}$).
	\item $\alpha\not\perp\beta$:
	$I=\langle e_2^+,f_2^-\rangle$ (type $d_{2,2}$).
	\item $\alpha\perp\beta$ and $\alpha\neq\beta$:
	$I=\langle e_2^+,2e_1^-+2f_1^-+x_-\rangle$
	(type $d_{2,3}$), where $x_-\in E_8(-2)$ and $x_-^2=-8$.
\end{itemize}
If $\operatorname{div}_{\Lambda^-}(w^-)=1$, then
$[w^-]_{\Lambda^-}=0$ and hence $\alpha=0$. The type-$a_2$ classification
allows us to take $w^-=e_1^-$, giving
\begin{itemize}
	\item $I=\langle e_2^+,e_1^-\rangle $  (type $d_{2,4}$).
\end{itemize}

\vspace{0.5em}
\noindent \textbf{Case $b$: $3$-dimensional isotropic subspaces
	$W=\langle w^+,w_1^-,w_2^-\rangle \subset \Lambda^+\oplus\Lambda^-$.}

Set
\[
\begin{aligned}
I^-&\coloneqq W\cap\Lambda^-
=\langle w_1^-,w_2^-\rangle,\\
P&\coloneqq
\psi^{-1}\!\left(
\left\langle
\left[\frac{w_1^-}{\operatorname{div}(w_1^-)}\right]_{\Lambda^-},
\left[\frac{w_2^-}{\operatorname{div}(w_2^-)}\right]_{\Lambda^-}
\right\rangle
\right)
\subseteq A_{\Lambda^+}.
\end{aligned}
\]
By the type-$d_1$ case, $I^-$
has divisor type $(1,2)$ or $(2,2)$. 
Again, from Lemma~\ref{lem:lattice_properties}(3), we only need to determine the class $[w^+/2]_{\Lambda^+}$.

If $I^-$ has divisor type $(1,2)$, we may take
$I^-=\langle e_2^-,e_1^-\rangle$; then
$P=\langle[e_2^+/2]_{\Lambda^+}\rangle$. There are $3$ orbits:
\begin{itemize}
	\item $[w^+/2]_{\Lambda^+}=[e_2^+/2]_{\Lambda^+}$:
	$w^+=e_2^+$, giving $W_{b_1}$.
	\item $[w^+/2]_{\Lambda^+}\not\perp P$:
	
	We may write $[w^+/2]_{\Lambda^+}=[\frac a2e_2^++\frac 12 f_2^++ \frac 12 x]_{\Lambda^+}$ for some $x\in E_8(-2)$.
	Applying the Eichler transvection $E_{e_2^+,\frac x2}$ reduces this to the case
	$x=0$ ($E_{e_2^+,\frac x2}$ is compatible with $E_{e_2^-,\frac {x'}2}$ as in \eqref{compatible}, which preserves $I^-$). Thus we may assume $[w^+/2]_{\Lambda^+}=[f_2^+/2]_{\Lambda^+}$ since $w^+ $ is isotropic, and
	$w^+=f_2^+$, giving $W_{b_2}$.
	\item $[w^+/2]_{\Lambda^+}\perp P$ and
	$[w^+/2]_{\Lambda^+}\neq[e_2^+/2]_{\Lambda^+}$:
	
	Now $[w^+/2]_{\Lambda^+}\neq0$ since $w^+$ is primitive.
	We may write $[w^+/2]_{\Lambda^+}=[\frac a2e_2^++ \frac 12 x]_{\Lambda^+}$ for some $x\in E_8(-2)$ with $x\notin 2\Lambda^+$.
	Applying $E_{e_2^+,\frac{x'}{2}}$ with $(x,x')=2$, we reduce to $a=0$.
	By Lemma~\ref{lem:lattice_properties}(1) and (3), we reduce to
	$w^+=2e_2^++2f_2^++x_+$, where $x_+^2=-16$, giving $W_{b_3}$.
\end{itemize}

If $I^-$ has divisor type $(2,2)$, take
\[
I^-=\langle e_2^-,u^-\rangle,
\qquad
u^-\coloneqq2e_1^-+2f_1^-+x_-,
\qquad x_-^2=-8.
\]
A similar calculation gives $3$ orbits:
\begin{itemize}
	\item $[w^+/2]_{\Lambda^+}\in P$: $w^+=e_2^+$, giving $W_{b_4}$.
	\item $[w^+/2]_{\Lambda^+}\notin P$ and $[w^+/2]_{\Lambda^+}\not\perp P$:
	$w^+=f_2^+$, giving $W_{b_5}$.
	\item $[w^+/2]_{\Lambda^+}\notin P$ and $[w^+/2]_{\Lambda^+}\perp P$:
	$w^+=2e_2^++2f_2^++x_+$, with $[x_+/2]_{\Lambda^+}$ orthogonal to
	$\psi^{-1}([x_-/2]_{\Lambda^-})$, giving $W_{b_6}$.
\end{itemize}

\vspace{0.5em}
\noindent \textbf{Case $c_1$: flags $\langle w^+\rangle \subset \langle w^+,w_1^-,w_2^-\rangle $ of dimensions $(1,3)$.}

For each $1\leq i\leq6$, the unique type-$c_1$ flag over $W_{b_i}$ is
$W_{b_i}\cap\Lambda^+\subset W_{b_i}$, giving the $6$ classes $c_{1,i}$.

\vspace{0.5em}
\noindent \textbf{Case $c_2$: flags $\langle w_1^-\rangle \subset \langle w^+,w_1^-,w_2^-\rangle $ of dimensions $(1,3)$.}

Suppose first that $I^-=W\cap\Lambda^-$ has divisor type $(1,2)$, so that
$I^-=\langle e_2^-,e_1^-\rangle$. 
From Lemma~\ref{lem:lattice_properties}(3), we know there are $2$ orbits of primitive $w\in I^-$ under
$\widetilde O(\Lambda^-)$, distinguished by the divisor of $w$. Consequently, for each
$i=1,2,3$, there are $2$ flags distinguished by their divisor type:
\[
c_{2,ia}:\ \langle e_1^-\rangle\subsetneq W_{b_i},
\qquad
c_{2,ib}:\ \langle e_2^-\rangle\subsetneq W_{b_i}.
\]

Now suppose that $I^-$ has divisor type $(2,2)$. Use the normalized
basis $I^-=\langle e_2^-,u^-\rangle$ above. 
Choose $x\in \Lambda^-$ such that $(u^-,x)=2$. The Eichler transvections $E_{e_2^-,x}\in \widetilde{O}(\Lambda^-)$ and $E_{f_2^-,u^-} \in \widetilde{O}(\Lambda^-)$ preserve $I^-$ and give
\[
ae_2^-+bu^-
\sim (a+2b)e_2^-+bu^-, \qquad 
ae_2^-+bu^-
\sim ae_2^-+(2a+b)u^-.
\]
This reduces the possible lines
to
\[
\langle e_2^-\rangle,
\qquad \langle u^-\rangle,
\qquad \langle e_2^-+u^-\rangle.
\]

For type $b_4$, we can further apply $E_{e_2^-, \frac x2}$ (compatible with $E_{e_2^+, \frac x2}$, which preserves $e_2^+$) to obtain $u^-\sim e_2^-+u^-$. Thus there are $2$ classes, represented by $w_1^-=u^-$ and $w_1^-=e^-_2$, and they are distinguished by
whether $[w_1^-/2]_{\Lambda^-}=\psi([w^+/2]_{\Lambda^+})$.

For type $b_5$, we can further apply $E_{f_2^-, \frac {u^-}2}$ to obtain $e_2^-\sim e_2^-+u^-$. Thus there are $2$ classes, represented by $w_1^-=u^-$ and $w_1^-=e^-_2$, and they are distinguished by whether
$b_{\Lambda^-}([w_1^-/2]_{\Lambda^-},\psi([w^+/2]_{\Lambda^+}))=0$.

For type $b_6$, both Eichler transvections above apply if we choose
$[u^-/2]_{\Lambda^-}$ orthogonal to $\psi([x_+/2]_{\Lambda^+})$. Thus there is only $1$ class.

Therefore $b_1,\ldots,b_5$ contribute $2$ classes each and $b_6$
contributes $1$, giving exactly $11$ type-$c_2$ classes.

Combining the $7$ cases gives
\[
1+2+6+6+11+2+4=32
\]
pairwise inequivalent rational isotropic flags, corresponding to $32$ Satake boundary components.
\qed
\section{Equivariant Gromov–Hausdorff Limits}

In this section, we specify the $\IZ_2$-action on each corresponding K3 limit.
\subsection{Period of the Flat Kummer Metric}

Let $X = (E_1 \times E_2)/\mathbb{Z}_2$ equipped with a flat orbifold metric $g$, where $E_i = \mathbb{C}/\Lambda_i$ are elliptic curves and the $\mathbb{Z}_2$ action is given by the standard involution $\iota: (z_1, z_2) \mapsto (-z_1, -z_2)$.
Write 
$\Lambda_i=\langle \lambda_i,\mu_i\rangle$. We define the involution:
\[ \sigma: (z_1, z_2) \mapsto (z_1 + \frac 12 \lambda_1, -z_2 + \frac 12 \lambda_2).
\]
One can readily verify that $\sigma$ descends to a fixed-point-free isometry on $X$.

Fix a marking $\phi:H^2(\widetilde{X},\mathbb Z)\to\Lambda$ of the minimal resolution
$\widetilde{X} \to X$ and use it to regard the cohomology classes as elements of
$\Lambda_{\mathbb C}$.

A flat product K\"ahler metric on $X$ has K\"ahler form and holomorphic
$2$-form
\[
\kappa=\frac{\sqrt{-1}}2
\left(dz_1\wedge d\bar z_1+ dz_2\wedge d\bar z_2\right),
\qquad
\omega=dz_1\wedge dz_2.
\]

One has $\sigma^*\kappa=\kappa$ and
$\sigma^*\omega=-\omega$.  Hence
$\kappa\in\Lambda^+_{\mathbb R}$ and
$[\omega]\in\Lambda^-_{\mathbb C}$.
Thus the period of the flat Kummer metric is
\[
(\kappa,[\omega])\in K\Omega_{\mathrm{En}} \quad \text{    with    } \quad	\Phi(\kappa,[\omega])
=(X,g)/\langle\sigma\rangle.\]

\subsubsection{Invariant and anti-invariant Kummer lattices}

Put $T=E_1\times E_2$, let
$q:T\to X=T/\langle\iota\rangle$ be the quotient.  Its singular points
are indexed by $a=(a_1,a_2,a_3,a_4)\in\mathbb F_2^4$ via
\[
p_a=\left(
\frac{a_1\lambda_1+a_2\mu_1}{2},
\frac{a_3\lambda_2+a_4\mu_2}{2}
\right).
\]
Let $k_a\in\Lambda$ be the class of the exceptional curve over $q(p_a)$;
then $(k_a,k_b)=-2\delta_{ab}$.  The Kummer lattice is
\[
\mathcal K
=\left(\bigoplus_{a\in\mathbb F_2^4}\mathbb Zk_a\right)_{\!\mathrm{sat}}
=\Lambda\cap\operatorname{span}_{\mathbb Q}\{k_a\}.
\]
It is generated by the $k_a$ and
$\frac12\sum_{a\in H}k_a$ for affine hyperplanes
$H\subset\mathbb F_2^4$; see \cite[Ch.~VIII, \S5]{barth-cptcplxsurface-04}.

Let $\theta_1,\ldots,\theta_4$ be dual to
$(\lambda_1,\mu_1,\lambda_2,\mu_2)$, and
write $e_{ij}$ for the image of $\theta_i\wedge\theta_j$ under
$H^2(T,\mathbb Z)(2)\hookrightarrow\Lambda$.  Then
\[
\begin{aligned}
	\mathcal K^\perp
	&=\langle e_{12},e_{34}\rangle
	\oplus\langle e_{13},-e_{24}\rangle
	\oplus\langle e_{14},e_{23}\rangle
	\cong U(2)^{\oplus3},\\
	\Lambda_{\mathbb Q}
	&=\mathcal K_{\mathbb Q}\oplus\mathcal K^\perp_{\mathbb Q}.
\end{aligned}
\]
Each displayed pair has intersection number $2$, and
$[\Lambda:\mathcal K\oplus\mathcal K^\perp]=2^6$.

Choose $\alpha_i=\lambda_i/2$ and set $\delta=(1,0,1,0)$.  The affine lift of
$\sigma$ acts on the $\theta_i$ with signs $(+,+,-,-)$ and sends
$p_a$ to $p_{a+\delta}$.  Hence
\[
\begin{aligned}
	&\sigma^*e_{12}=e_{12},\qquad \sigma^*e_{34}=e_{34},\\
	&\sigma^*e_{ij}=-e_{ij}
	\quad(i\in\{1,2\},\ j\in\{3,4\}),\\
	&\sigma^*k_a=k_{a+\delta}.
\end{aligned}
\]
Writing $\Lambda^\pm=\{v\in\Lambda:\sigma^*v=\pm v\}$, and choosing
$1$ representative from each orbit of $a\mapsto a+\delta$ to form $R$,
we obtain
\[
\begin{aligned}
	\Lambda^+_{\mathbb Q}
	&=\operatorname{span}_{\mathbb Q}
	\{e_{12},e_{34},k_a+k_{a+\delta}:a\in R\},\\
	\Lambda^-_{\mathbb Q}
	&=\operatorname{span}_{\mathbb Q}
	\{e_{13},e_{14},e_{23},e_{24},k_a-k_{a+\delta}:a\in R\}.
\end{aligned}
\]
The integral eigensublattices are the intersections of these spaces with
$\Lambda$.  As recalled in Section~2,
\[
\Lambda^+\cong U(2)\oplus E_8(-2),
\qquad
\Lambda^-\cong U\oplus U(2)\oplus E_8(-2),
\]
so, for $A_M=M^\vee/M$,
$A_{\Lambda^+}\cong A_{\Lambda^-}\cong(\mathbb Z/2\mathbb Z)^{10}$.
As in Section~4, unimodularity of $\Lambda$ gives the anti-isometry
\[
\psi:(A_{\Lambda^+},q_{\Lambda^+})
\xrightarrow{\sim}(A_{\Lambda^-},-q_{\Lambda^-}),
\]
characterized by
\[
\psi([v_+])=[v_-]
\quad\Longleftrightarrow\quad
v_++v_-\in\Lambda,
\qquad v_\pm\in(\Lambda^\pm)^\vee.
\]

For $i<j$, set
\[
S_{ij}=\sum_{a_i=a_j=0}k_a.
\]
The proper-transform formula \cite[p.~319]{barth-cptcplxsurface-04} gives
\[
g_{ij}:=\frac{e_{ij}-S_{ij}}2\in\Lambda.
\]
This is precisely the Poincar\'e dual class of the proper transform of the
corresponding coordinate $2$-torus.  Since $\sigma$ translates the indices by
$\delta=(1,0,1,0)$,
\[
S_{23}+\sigma^*S_{23}
=S_{12}+\sigma^*S_{12}=\sum_{a_2=0}k_a.
\]
Thus $\sigma^*(S_{23}-S_{12})=-(S_{23}-S_{12})$, and hence
$e_{23}+S_{23}-S_{12}\in\Lambda^-$.  Moreover,
\[
\frac{e_{12}}2+\frac{e_{23}+S_{23}-S_{12}}2
=g_{12}+g_{23}+S_{23}\in\Lambda.
\]
The $2$ summands lie in $(\Lambda^+)^\vee$ and $(\Lambda^-)^\vee$, respectively; therefore,
\[
\psi\left(\left[\frac{e_{12}}2\right]\right)
=\left[\frac{e_{23}+S_{23}-S_{12}}2\right].
\]
The same argument, with $(e_{12},e_{23},S_{12},S_{23})$ replaced by
$(e_{34},e_{14},S_{34},S_{14})$, gives
\[
\psi\left(\left[\frac{e_{34}}2\right]\right)
=\left[\frac{e_{14}+S_{14}-S_{34}}2\right].
\]

We also note that $g_{24}\in \Lambda^-$; thus $e_{13}$ has divisor $1$, while the other $e_{ij}$ have divisor $2$ in $\Lambda^\pm$.

\subsection{Degenerations of Kummer Orbifolds} \label{section-kummer-example}
Here we consider degenerations of this Kummer orbifold in various dimensions as parameters approach zero. We write $\CM(a_{i,j})$ for the corresponding Satake boundary components below.

\vspace{0.5em}

\noindent\textbf{3-Dimensional Collapsing:}
\begin{enumerate}
	\item \textbf{$\lambda_1 \to 0$:} Write
	$z_1=x_1\lambda_1+y_1\mu_1$, with $x_1,y_1\in\mathbb R/\mathbb Z$.
	After the $\lambda_1$-circle collapses,
	\[
	\begin{aligned}
		X_\infty
		&=\bigl((\mathbb R/\mathbb Z)_{y_1}\times E_2\bigr)
		/\langle\iota_\infty\rangle\cong T^3/\IZ_2,\\
		\iota_\infty(y_1,z_2)&=(-y_1,-z_2),\\
		\sigma_\infty[y_1,z_2]
		&=[y_1,-z_2+\tfrac12\lambda_2].
	\end{aligned}
	\]
	The class $[y_1,z_2]$ is fixed iff
	$(y_1,-z_2+\lambda_2/2)$ equals $(y_1,z_2)$ or
	$(-y_1,-z_2)$.  The latter would imply $\lambda_2/2=0$ in $E_2$;
	hence only the former occurs, giving
	\[
	2z_2=\frac{\lambda_2}{2},\qquad
	z_2=\frac{\lambda_2}{4}
	+\frac{\varepsilon\lambda_2+\eta\mu_2}{2},
	\quad \varepsilon,\eta\in\{0,1\}.
	\]
	The $4$ solution circles are paired by $\iota_\infty$ according to
	$\varepsilon=0,1$.  Thus
	\[
	\operatorname{Fix}(\sigma_\infty)
	=S^1\sqcup S^1.
	\]
	Here
	$(w,e_{23},e_{24},e_{34})\to0$, so the limit lies in $\CM(b_5)$.
	\item \textbf{$\mu_1 \to 0$:} A similar calculation shows that the limit exhibits a free action on $T^3/\mathbb{Z}_2$.
	We have $(w, e_{13}, e_{14}, e_{34}) \to 0$, and the limit lies in $\CM(b_3)$.
\end{enumerate}

\noindent\textbf{2-Dimensional Collapsing:}
\begin{enumerate}
	\item \textbf{$\mu_1, \mu_2 \to 0$:} The limit space is $\mathbb{R}P^2$.
	Here $(w, e_{13}) \to 0$; if we further assume $\mu_1\sim \mu_2$, then the period limit lies in $\CM(a_{2,1})$.
	\item \textbf{$\lambda_1, \lambda_2 \to 0$:} The limit space is a square.
	Here $(w, e_{24}) \to 0$; if $\lambda_1\sim \lambda_2$, then the period limit lies in $\CM(a_{2,2})$.
	\item \textbf{$\lambda_1, \mu_1 \to 0$:} The limit space is $\mathbb{P}^1$. Here $(w, e_{34}) \to 0$; if $\lambda_1\sim \mu_1$, then the period limit lies in $\CM(a_{1,1})$.
\end{enumerate}

\noindent\textbf{1-Dimensional Collapsing with limit $I^1$:}
\begin{enumerate}
	\item \textbf{$\lambda_1, \lambda_2, \mu_1 \to 0$:} We have $(w, e_{14}, e_{24}, e_{34}) \to 0$, corresponding to $\CM(b_6)$.
	\item \textbf{$\lambda_2, \mu_1, \mu_2 \to 0$:} We have $(w, e_{13}, e_{14}, e_{12}) \to 0$, corresponding to $\CM(b_2)$.
\end{enumerate}

This construction covers every point of $\CM(b_2)$, $\CM(b_3)$, $\CM(b_5)$, and $\CM(b_6)$ and special points of $\CM(a_{1,1})$, $\CM(a_{2,1})$, and $\CM(a_{2,2})$, thereby determining their Gromov--Hausdorff limits.

For the $\CM(b_i)$ and $\CM(c_{i,j})$ boundary components, let $I_3$ be the corresponding primitive lattice of rank $3$ in $\Lambda_{K3}$. Then the dimension of the GH limit is determined by the type of $(I_3)^\perp/I_3$, which can be $(-E_8)^2$ or $-\Gamma_{16}$, with $(-E_8)^2$ corresponding to an $I^1$ limit. 
For $\CM(b_1)$, the quotient $W^\perp/W$ is $E_8(-2)\oplus E_8(-2)$ inside $L^+\oplus L^-$, hence $(-E_8)^{\oplus2}$ in $L$.
Consequently, the geometric limit is $I^1/\IZ_2$ and remains $I^1$ after scaling.

\subsection{$\mathbb{Z}_2$ Action on Generalized K3 Metrics} \label{section-5.3}

\begin{lemma}\label{lem:involution_types}
	Every conformal involution $\sigma$ of $\IP^1=\IC\cup \{\infty\}$ is of one of the following types.
	
	\textbf{Holomorphic:}
	\begin{enumerate}
		\item $\sigma = \mathrm{Id}$.
		\item $\sigma$ has exactly $2$ fixed points. Locally, it acts as a $180^\circ$ rotation, and the quotient $X/\sigma$ is homeomorphic to $\mathbb{P}^1$.
	\end{enumerate}
	
	\textbf{Anti-holomorphic:}
	\begin{enumerate}
		\setcounter{enumi}{2}
		\item The fixed locus is an embedded $S^1$. $\sigma$ acts as a reflection near this locus, and $X/\sigma \simeq \mathbb{D}^2$.
		\item There is no fixed locus, and $X/\sigma \simeq \mathbb{R}P^2$.
	\end{enumerate}
\end{lemma}

\begin{proof}
	\begin{enumerate}
		\item \textbf{Holomorphic Case:} 
		Assume $\sigma\neq \mathrm{Id}$.
		By applying a suitable M\"obius transformation, we may assume $\sigma$ exchanges $0$ and $\infty$. Then $\sigma =\frac{1}{az}$ and is conjugate to $\sigma=\frac{1}{z}$.
		\item \textbf{Anti-holomorphic case.} 
		Assume $\sigma$ exchanges $0$ and $\infty$.
		Thus $\sigma(z) = \frac{1}{a\bar{z}}$ for some $a \in \mathbb{C}^*$, and 
		$\sigma^2 = \mathrm{Id}$ implies $a \in \mathbb{R}$.
		After further conjugation by a dilation, $\sigma(z)$ is conjugate to $\pm \frac{1}{\bar{z}}$.

	\end{enumerate}
\end{proof}

Recall that a $2$-dimensional K3 GH limit is a generalized K\"ahler--Einstein metric on $\IP^1$, and such a metric determines the complex structure on $\IP^1$ up to conjugation. Therefore, every isometric involution on a $2$-dimensional K3 GH limit has one of these $4$ types. We now treat the $\CM(a_{i,j})$ cases of Theorem~\ref{thm-quotient-type}.

\begin{proposition} \label{prop:constant_topo_type}
	In each connected boundary component $\CM(a_{i,j})\subset \overline{\mathcal{M}_{\mathrm{En}}}$, where $(i,j)\in\{(1,1),(2,1),(2,2)\}$, the involution $\sigma$ on $\mathbb{P}^1$ has a constant type given by Lemma~\ref{lem:involution_types} and is thus completely determined by the Kummer examples in the previous section.
\end{proposition}

\begin{proof}
	
	We can decompose $\CM(a_{i,j}) = \bigcup_{k=1}^4 A_k$, where $A_k$ denotes the locus of metrics exhibiting type $k$. Because $\CM(a_{i,j})$ is connected, we need only show that each $A_k$ is closed.

	Suppose $\xi_i \to\xi_\infty$ in $\CM(a_{i,j})$. Let $f(\xi_i)$ be its image in $\overline{\CM}_{K3}^{\mathrm{Sat,ad}}$, and let $(\IP^1,g_i)= \Phi_{K3}(\xi_i)$. From the Weierstrass description in \cite{ouyang2025compactificationmetricmodulispace}, we may assume that $g_i$ converges to $g_\infty$ outside a finite set of points. Thus, after passing to a subsequence, we may assume that the involution $\sigma_i$ converges to an involution $\sigma_\infty$ on $\IP^1$.

	It suffices to prove the following claim.
	
	\textbf{Claim:}
	If all $\sigma_i$ are of type $k$, then $\sigma_\infty$ is also of type $k$.

	If $\sigma_i$ is of type (1), then the conclusion holds trivially. If $\sigma_i$ is of type (2), consider a sequence of fixed points $p_i \to p_\infty$. If the tangent map converges as $-\mathrm{Id} \to -\mathrm{Id}$, then the limit has an isolated singularity. A similar argument applies to type (3). It remains to deal with the case in which $\sigma_i$ is of type (4).

	Assume to the contrary that the limit $\sigma_\infty$ is not of type (4). Then $\sigma_\infty$ must be of type (3).
	We may write each $\sigma_i$ in our sequence as $\sigma_i = f_i \circ \sigma_0 \circ f_i^{-1}$, where $\sigma_0(z) = -\frac{1}{\bar{z}}$.
	Representing this in $PGL_2(\mathbb{C})$, let $f_i(z)=\frac{az+b}{cz+d}$ correspond to the matrix $A = \begin{pmatrix} a & b \\ c & d \end{pmatrix}$.
	The conjugated involution is then given by:
	\begin{equation*}
		\sigma_i = A \begin{pmatrix} 0 & -1 \\ 1 & 0 \end{pmatrix} \bar{A}^{-1} = \begin{pmatrix} b\bar{d} + a\bar{c} & -|b|^2 - |a|^2 \\ |d|^2 + |c|^2 & -b\bar{d} - c\bar{a} \end{pmatrix}.
	\end{equation*}
	Assume that $\sigma_i \to \sigma_\infty$ and $\sigma_\infty(0)=0$.
	Then, for arbitrarily large $N > 0$, as $i \to \infty$, either $|b\bar{d}| \geq N(|b|^2+|a|^2)$ or $|c\bar{a}| \geq N(|b|^2+|a|^2)$.
	
	We assume the first case holds (the second follows by an identical symmetric argument).
	Under this assumption, we have:
	\begin{equation*}
		|d| > 2N \left( |b| + \left| \frac{a^2}{b} \right| \right) \geq N(|b| + |a|).
	\end{equation*}
	Squaring this and examining the lower-left entry of $\sigma_i$, we find:
	\begin{equation*}
		|d|^2 + |c|^2 = \frac{1}{2}|d|^2 + \frac{1}{2}|d|^2 + |c|^2 \geq \frac{N}{2}|bd|
		+ \frac{N^2}{2}|a|^2 + |c|^2 \geq \frac{N}{2}(|bd| + |ac|).
	\end{equation*}
	This implies that as $i \to \infty$, the lower-left term dominates, forcing $\sigma_i(1) \to 0$.
	However, this contradicts the injectivity of the limiting map $\sigma_\infty$.
	
	Thus, the type of the involution cannot jump in the limit, and the proposition is proved.
\end{proof}

For $\CM(c_{i,j})$, we know that it lies in the closure of $\CM(b_j)$ and $\CM(a_{i,k})$.
Let $I_{c_{i,j}}$ be the corresponding primitive lattice of rank $3$ in $\Lambda_{K3}$. Then the GH limit is $I^1$ or a $\IZ_2$-quotient of $\IP^1$ and depends on whether $(I_{c_{i,j}})^\perp/I_{c_{i,j}}$ is isomorphic to $(-E_8)^2$ or $-\Gamma_{16}$. In the $\IP^1$ case, choose $\xi_i$ in $\CM(a_{\ell,m})$ converging to $\xi_\infty$ in $\CM(c_{i,j})$.
From \cite[Appendix]{ouyang2025compactificationmetricmodulispace}, we can still derive the smooth convergence of generalized KE metrics outside finitely many points. Thus Proposition~\ref{prop:constant_topo_type} applies to give the $\IZ_2$-quotient type. In particular, the limit space 
 is of the same type as in $\CM(a_{i,k})$.

\subsection{3-Dimensional Collapse over $\CM(b_4)$}\label{section-5.4}

Put
\[
u^-:=2e_1^-+2f_1^-+x_-,\qquad x_-^2=-8,
\]
so that
$W_{b_4}=\langle e_2^-,u^-,e_2^+\rangle_{\mathbb Q}$.
The compatibility of the gluing
$\Lambda^+\oplus\Lambda^-\subset\Lambda$ gives
$\psi([e_2^+/2])=[e_2^-/2]$, and hence
$(e_2^++e_2^-)/2\in\Lambda$.  Moreover, the sublattice
\begin{equation*}
	I_{b_4}:=W_{b_4}\cap \Lambda_{K3}=\left\langle
	e_2^+,\ \frac{e_2^++e_2^-}{2},\ u^-
	\right\rangle_{\mathbb Z} \subset \Lambda_{K3},
\end{equation*}
is primitive in $\Lambda_{K3}$. That is, $I_{b_4}$ contains all the integral points in $W_{b_4}$.

Recall that along the $\CM(b_i)$ boundary components, the corresponding K3 metrics collapse to
dimension $1$ or $3$, depending on whether $I_{b_i}^\perp/I_{b_i}$ is isomorphic to $-\Gamma_{16}$ or $(-E_8)^2$. In our case, it is easy to see that $I_{b_4}^\perp/I_{b_4}\cong I_{b_5}^\perp/I_{b_5}$, and we know that it corresponds to a $3$-dimensional limit from the Kummer example in the previous section.

Suppose $(\kappa_i, [\omega_i])$ represents a sequence of points converging to $\CM(b_4)$. After choosing a suitable marking for $H^2(X,\IZ)$, the limiting $T^3/\IZ_2$ is given by the inner products of $\kappa_i$ and $\omega_i$ with the $3$ primitive basis vectors of $I_{b_4}$. The torus $T^3=\IR^3/\langle v_1,v_2,v_3\rangle$ is given by
\begin{equation}\label{eq:b4-torus-periods}
	\begin{pmatrix} v_1 \\ v_2 \\ v_3 \end{pmatrix}
	=
	\begin{pmatrix}
		2x & 0 & 0 \\
		x  & \ast & \ast \\
		0  & \ast & \ast
	\end{pmatrix}.
\end{equation}

Conversely, all the $T^3/\IZ_2$ spaces of the above form appear as K3 limits in the Enriques boundary $b_4$.

\begin{proposition}\label{prop:T3_Z2_quotients}
	Let $T^3=\mathbb R^3/\Lambda_{T^3}$, where
	$\Lambda_{T^3}=\langle v_1,v_2,v_3\rangle_{\mathbb Z}$ has period
	matrix \eqref{eq:b4-torus-periods}, and put
	$X=T^3/\mathbb Z_2$, where the nontrivial element acts by
	$\iota[x]=[-x]$.  Assume that
	$T^3$ is general in this family, so that
	\[
	O(\Lambda_{T^3})=\{\pm I,\pm R\},
	\]
	where $R$ is the linear involution determined by
	\[
	R(v_1)=v_1,\qquad R(v_2)=v_1-v_2,\qquad R(v_3)=-v_3.
	\]
	Set $V:=\Lambda_{T^3}/2\Lambda_{T^3}$.  For
	$\lambda\in\Lambda_{T^3}$, we also write $\lambda$ for its class in $V$
	and define
	\[
	\tau_{\lambda}[x]
	=\left[x+\frac{\lambda}{2}\right],\qquad
	\eta_{\lambda}[x]
	=\left[Rx+\frac{\lambda}{2}\right].
	\]
	Then every isometry $\sigma$ of $X$ satisfying $\sigma^2=\mathrm{Id}$
	is exactly one of
	\[
	\tau_{\lambda}\quad(\lambda\in V),
	\qquad
	\eta_{\lambda}\quad
	\left(\lambda\in
	\langle v_1,v_3\rangle_{\mathbb F_2}\right).
	\]
	Up to conjugacy in
	$\operatorname{Isom}(X)$, there are $7$ nontrivial involutions, represented
	by the following list, together with the identity:
	\begin{enumerate}
		\item $\eta_{v_3}$, with
		$\eta_{v_3}\sim\eta_{v_1+v_3}$;
		\item $\eta_0$, with $\eta_0\sim\eta_{v_1}$;
		\item $\tau_{v_1}$;
		\item $\tau_{v_3}$;
		\item $\tau_{v_1+v_3}$;
		\item $\tau_{v_2}$, with
		$\tau_{v_2}\sim\tau_{v_1+v_2}$;
		\item $\tau_{v_2+v_3}$, with
		$\tau_{v_2+v_3}\sim\tau_{v_1+v_2+v_3}$.
		\item $\tau_0=\mathrm{Id}$.
	\end{enumerate}
\end{proposition}

\begin{proof}
	Write
	\[
	\Gamma=\Lambda_{T^3}\rtimes\{\pm I\},
	\qquad X=\mathbb R^3/\Gamma.
	\]
	Every isometry of $X$ lifts to an affine Euclidean isometry
	$F(x)=Qx+\lambda/2$ normalizing $\Gamma$.  Conjugating the translations in
	$\Gamma$ by $F$ shows that $Q\Lambda_{T^3}=\Lambda_{T^3}$, while
	\[
	F\circ(-I)\circ F^{-1}(x)=-x+\lambda
	\]
	shows that $\lambda\in\Lambda_{T^3}$.  Conversely, these $2$ conditions
	are sufficient for $F$ to descend to $X$, and the class
	$\lambda\in V$ is uniquely determined.
	
	The lifts $(Q,\lambda/2)$ and $(-Q,-\lambda/2)$ induce the same
	isometry of $X$.
	By the genericity assumption, we may therefore take $Q=I$ or $Q=R$.
	Moreover,
	\[
	F^2(x)=Q^2x+\frac{(I+Q)\lambda}{2},
	\]
	so the induced isometry has square equal to the identity precisely when
	\[
	(I+Q)\lambda\in2\Lambda_{T^3}.
	\]
	For $Q=I$ this condition is automatic, giving all $8$ translations
	$\tau_{\lambda}$.  For $Q=R$, reduction modulo
	$2\Lambda_{T^3}$ gives
	\[
	R(v_1)=v_1,\qquad
	R(v_2)=v_1+v_2,\qquad
	R(v_3)=v_3
	\qquad\text{in }V.
	\]
	Thus $(I+R)\lambda=0$ in $V$ if and only if
	$\lambda\in\langle v_1,v_3\rangle_{\mathbb F_2}$,
	which gives the $4$ involutions of the form $\eta_{\lambda}$.
	
	It remains only to determine conjugacy.  Conjugation by
	$\eta_0$ sends $\tau_{\lambda}$ to
	$\tau_{R\lambda}$.  The resulting $6$ orbits in $V$ are
	\[
	\{0\},\ \{v_1\},\ \{v_3\},\ \{v_1+v_3\},\
	\{v_2,v_1+v_2\},\
	\{v_2+v_3,v_1+v_2+v_3\}.
	\]
	For the $R$-type involutions, conjugation by
	$\tau_\mu$ changes the parameter by $(I-R)\mu$.  Since
	$\operatorname{Im}(I-R)=\langle v_1\rangle$ in $V$, their $2$
	conjugacy classes are represented by $\eta_0$ and
	$\eta_{v_3}$.  These are exactly the classes listed above.
	\end{proof}
	
	\begin{theorem}\label{thm:b4-quotient}
		For every $\xi\in\CM(b_4)$, we have $\overline{\Phi}(\xi)=X_\xi/\langle\eta_{v_3}\rangle$, where $X_\xi=T_\xi^3/\IZ_2$ is the corresponding K3 limit.
	
	\end{theorem}
	
	\begin{proof}
		We divide the proof into $4$ steps.
		
		\smallskip
		\noindent
		1.
		For $\lambda\in\{0,v_3\}$, let $A_{\eta_\lambda}$ be the locus of
		$\xi\in\CM(b_4)$ for which $\overline{\Phi}(\xi)$ is isometric to
		$X_\xi/\langle\eta_\lambda\rangle$, and put
		$A_\tau=\bigcup_\lambda A_{\tau_\lambda}$, where $\lambda$ ranges over
		the $6$ representatives in Proposition~\ref{prop:T3_Z2_quotients}.
		
		Each $A_{\eta_\lambda}$ and $A_{\tau_\lambda}$ is closed. Indeed, if
		$\xi_i$ belongs to one of these loci and $\xi_i\to\xi$, the explicit formulas for
		$\eta_\lambda$ and $\tau_\lambda$ show that the corresponding isometries
		converge equivariantly to the isometry with the same formula on $X_\xi$.
		The continuity of $\overline{\Phi}$ then shows that $\overline{\Phi}(\xi)$ is the corresponding
		quotient.
		
		The general locus is dense in $\CM(b_4)$. Hence
		Proposition~\ref{prop:T3_Z2_quotients} shows that 
		\[
		\CM(b_4)=A_{\eta_0}\cup A_{\eta_{v_3}}\cup A_\tau.
		\]
		In particular, at a nongeneral point the quotient is still isometric to
		one of the $8$ types in Proposition~\ref{prop:T3_Z2_quotients}.
		
		\smallskip
		\noindent
		2.
		Fix $X=T^3/\langle\iota\rangle$, and write
		$x=av_1+bv_2+cv_3$ with $a,b,c\in\mathbb R/\mathbb Z$. A point
		$[x]\in X$ is fixed by the involution induced by
		$F(x)=Qx+\lambda/2$ precisely when
		\[
		F(x)=x\quad\text{or}\quad F(x)=-x\qquad\bmod\Lambda_{T^3}.
		\]
		For $\tau_\lambda$, we have $Q=I$. If $\lambda\ne0$, the first equation
		has no solution, while the second is
		\[
		2x+\frac{\lambda}{2}=0\qquad\bmod\Lambda_{T^3}.
		\]
		It has $8$ solutions in $T^3$, paired by $\iota$, so
		$\operatorname{Fix}_X(\tau_\lambda)$ consists of $4$ isolated regular
		points. For $\lambda=0$, $\tau_0$ is the identity.
		
		For $\eta_\lambda$, we have $Q=R$. Let $q:T^3\to X$ be the quotient map. Since
		$Rx=(a+b)v_1-bv_2-cv_3$, a direct calculation gives
		\[
		\begin{aligned}
		\operatorname{Fix}_X(\eta_0)
		&=q\!\left(\{b=0,\ c\in\{0,\tfrac12\}\}\cup\{2a+b=0\}\right),\\
		\operatorname{Fix}_X(\eta_{v_3})
		&=q\!\left(\{b=0,\ c\in\{\tfrac14,\tfrac34\}\}\right).
		\end{aligned}
		\]
		Here $\{2a+b=0\}$ is a $2$-dimensional subtorus, whereas the other
		components are circles; the $2$ circles for $\eta_{v_3}$ are interchanged
		by $\iota$ and hence give a single circle in $X$. 
		
		Thus the quotients by
		$\eta_0$, $\eta_{v_3}$, and a $\tau$-type involution have different types of fixed points, and hence their quotient spaces are not isometric. Consequently,
		\[
		\CM(b_4)=A_{\eta_0}\sqcup A_{\eta_{v_3}}\sqcup A_\tau.
		\]
		All $3$ sets are closed by step 1. Since $\CM(b_4)$ is connected, it must equal exactly
		one of them.
		
		\smallskip
		\noindent
		3.
		Assume that $\CM(b_4)=A_\tau$, and choose
		$\xi_i\in\CM(b_4)$ converging to a point
		$\xi_\infty\in\CM(c_{2,4a})$ (the same argument applies to
		$\CM(c_{2,4b})$). We know $\Phi(\xi_\infty)$ is a disk in \S~\ref{section-5.3}.
		
		On the other hand, the K3 limits $\overline{\Phi}_{K3}(\xi_i)$ collapse to
		$X_\infty=T^2/\langle-1\rangle\cong\IP^1$. After passing to a subsequence, the
		$\tau$-actions converge to an involution of $X_\infty$ whose linear part,
		modulo $\{\pm I\}$, is the identity. It is therefore orientation
		preserving, so its quotient is homeomorphic to $\IP^1$, a contradiction.
		
		\smallskip
		\noindent
		4.
		It remains to exclude $\eta_0$. Choose points of $\CM(b_4)$ converging to
		$\CM(c_{1,4})$ by collapsing the $v_1$-direction. The linear map $R$
		then induces $-I$ on the limiting $T^2$, so $\eta_0$ induces the identity
		on $T^2/\langle-1\rangle\cong\IP^1$. 
		On the other hand, the quotient type in $\CM(c_{1,4})$ is the same as in
		$\CM(a_{1,1})$ and arises from a
		nontrivial involution with $2$ fixed points, giving
		a contradiction. Hence $\eta_{v_3}$ is the unique type on
		$\CM(b_4)$.
	\end{proof}

\subsection{$\mathbb{Z}_2$ Quotient on the Limit}

Combining Sections~\ref{section-kummer-example}--\ref{section-5.4}, we have proved the following.
\begin{theorem}\label{thm-quotient-type}
	At each boundary point, the Gromov--Hausdorff limit is $Y_\infty=X_\infty/\mathbb Z_2$, where $X_\infty$ is the corresponding K3 limit. Its topology is as follows.
	
	\begin{itemize}
		\item \textbf{$\CM(a_{1,1})$:} $X_\infty=\IP^1$; the involution has $2$ fixed points, and $Y_\infty\cong\IP^1$.
		
		\item \textbf{$\CM(a_{2,1})$:} $X_\infty=\IP^1$; the involution is free, and $Y_\infty\cong\IR\IP^2$.
		
		\item \textbf{$\CM(a_{2,2})$:} $X_\infty=\IP^1$; the fixed locus is a circle, and $Y_\infty$ is a disk.
		
		\item \textbf{$\CM(b_3)$:} $X_\infty=T^3/\IZ_2$; the involution is free, and the quotient is the second three-dimensional Kummer degeneration in Section~\ref{section-kummer-example}.
		
		\item \textbf{$\CM(b_4)$:} $X_\infty=T^3/\IZ_2$; the involution is $\eta_{v_3}$, as proved in Theorem~\ref{thm:b4-quotient}.
		
		\item \textbf{$\CM(b_5)$:} $X_\infty=T^3/\IZ_2$; the fixed locus is the disjoint union of $2$ circles, and the quotient is the first three-dimensional Kummer degeneration in Section~\ref{section-kummer-example}.
		
		\item \textbf{$\CM(b_1)$, $\CM(b_2)$, and $\CM(b_6)$:} The quotient space is the unit segment $I^1$.
		
		\item \textbf{$\CM(c_{i,j})$:} The quotient space is either $I^1$ or $2$-dimensional. Its specific topological type is determined by the adjacent $\CM(a_1)$, $\CM(a_2)$, and $\CM(b)$ boundary components.
		
		\item \textbf{$\CM(d_{i,j})$:} The quotient space is the unit segment $I^1$.
	\end{itemize}
\end{theorem}

\begin{proof}[Proof of Corollary~\ref{coro-fix-classes}]
	It suffices to determine which Satake boundary strata meet the closures of
	the $2$ restricted period loci in $\overline{\CM\en}^{\mathrm{Sat}}$.
	Choose a marking and write a period point
	as $(\kappa,[\omega])\in
	O(1,9)/O(9)\times O(2,10)/(O(2)\times O(10))$.
	
	First fix the complex structure, hence fix $[\omega]=[\omega_J]$, and let
	$\CM_J$ denote the locus obtained by varying $\kappa$.  Only the
	$O(1,9)/O(9)$ factor can diverge.  Its rational Satake boundary consists of
	isotropic lines $I^+\subset\Lambda^+_{\mathbb Q}$.  By Theorem~\ref{thm:isotropic_classes}, every such line determines a boundary point
	in $\CM(a_{1,1})$, and hence
	\[
	\partial\overline{\CM_J}^{\mathrm{Sat}}
	\subset \CM(a_{1,1}).
	\]
	
	Let $p:X\to Y$ be the K3 cover, with deck
	involution $\sigma$, and 
	let $F\in H^2(X,\mathbb Z)$ be the cohomology class corresponding to $e_2^+$.  Then $F$ is a primitive nef
	isotropic class in $\operatorname{Pic}(X)$.  By
	\cite[Ch.~2, Prop.~3.10, p.~31]{huybrechts-lecturenotes}, $|F|$ is a
	base-point-free pencil and defines an elliptic fibration
	\[
	\pi_E:X\longrightarrow B\cong\IP^1.
	\]
	
	Since $\sigma^*F=F$, the involution $\sigma$ acts on
	$\IP H^0(X,\mathcal O_X(F))\cong B$ and therefore induces an involution
	$\overline{\sigma}$ of $B$ such that
	\[
	\begin{array}{ccc}
		X & \xrightarrow{\ \sigma\ } & X\\
		{\scriptstyle\pi_E}\downarrow && \downarrow{\scriptstyle\pi_E}\\
		B & \xrightarrow{\ \overline{\sigma}\ } & B
	\end{array}
	\]
	commutes.  Moreover, $\overline{\sigma}$ is nontrivial and hence fixes $2$ points, as in Lemma~\ref{lem:involution_types}.  Indeed, if it
	were the identity, then $\sigma$ would act on a general elliptic
	fiber as a translation. It would
	then preserve both the relative holomorphic $1$-form and the base
	coordinate, and consequently the holomorphic $2$-form of $X$.  This
	contradicts $\sigma^*\omega_X=-\omega_X$.
	
	Quotienting the preceding diagram gives an elliptic fibration
	\[
	Y=X/\langle\sigma\rangle
	\xrightarrow{\ f_E\ } 
	B/\langle\overline{\sigma}\rangle\cong\IP^1.
	\]
	Finally, the results for K3 collapsing \cite{gross-tosatti-zhang-16,hein-tosatti-15} apply $\sigma$-equivariantly to the generalized K\"ahler--Einstein metric $(\IP^1,g_\infty)$ given by the elliptic fibration.
	
	Now fix a polarized K\"ahler class $\kappa\in\Lambda^+$ and let
	$\CM_\kappa$ be the locus obtained by varying $[\omega]$.  The
	$O(1,9)/O(9)$ factor is then fixed. The
	classification in Theorem~\ref{thm:isotropic_classes} gives
	\[
	\partial\overline{\CM_\kappa}^{\mathrm{Sat}}
	\subset
	\CM(a_{2,1})\cup\CM(a_{2,2})
	\cup\CM(d_{1,1})\cup\CM(d_{1,2}).
	\]
	
	Theorem~\ref{thm-quotient-type} gives
	\[
	\CM(a_{2,1})\longrightarrow  \IR\IP^2,\qquad
	\CM(a_{2,2})\longrightarrow D^2,\qquad
	\CM(d_{1,1})\cup\CM(d_{1,2})\longrightarrow I^1.
	\]
	These are exactly the asserted limits for a fixed polarized K\"ahler
	class.
\end{proof}

\bibliography{enrique-compactification}

\end{document}